\documentclass[11pt,a4paper,reqno]{amsart}

\usepackage{epsfig}
\usepackage{amsfonts}
\usepackage{amsmath}
\usepackage{amssymb}
\usepackage{amsthm}
\usepackage{color}
\usepackage{enumerate}
\usepackage[T1]{fontenc}
\usepackage[latin1]{inputenc}
\usepackage{ae,aecompl}
\usepackage{graphicx}
\usepackage{nicefrac}
\usepackage[includehead,includefoot,margin=25mm]{geometry}
\usepackage{bigstrut}
\usepackage{setspace}

\newtheorem{theorem}{Theorem}[section]
\newtheorem{corollary}[theorem]{Corollary}
\newtheorem{lemma}[theorem]{Lemma}
\newtheorem{proposition}[theorem]{Proposition}

{\theoremstyle{remark}
  \newtheorem{remark}[theorem]{Remark}}
{\theoremstyle{definition}
  \newtheorem{definition}[theorem]{Definition}

  \newtheorem{example}[theorem]{Example}

}

\newcommand{\PP}[0]{\ensuremath{\mathbb{P}}}

\newcommand{\ZZ}[0]{\ensuremath{\mathbb{Z}}}
\newcommand{\GA}[0]{\ensuremath{\mathbb{G}_{\mathrm{a}}}}
\newcommand{\GM}[0]{\ensuremath{\mathbb{G}_{\mathrm{m}}}}
\newcommand{\AF}[0]{\ensuremath{\mathbb{A}}}

\newcommand{\RR}[0]{\ensuremath{\mathcal{R}}}
\newcommand{\QQ}[0]{\ensuremath{\mathbb{Q}}}
\newcommand{\TT}[0]{\ensuremath{\mathbb{T}}}
\newcommand{\KK}[0]{\ensuremath{\mathbf{k}}}
\newcommand{\OO}[0]{\ensuremath{\mathcal{O}}}

\newcommand{\DD}[0]{\ensuremath{\mathfrak{D}}}
\newcommand{\tDD}[0]{\ensuremath{\widetilde{\mathfrak{D}}}}

\newcommand{\tN}[0]{{\ensuremath{\widetilde{N}}}}
\newcommand{\tM}[0]{{\ensuremath{\widetilde{M}}}}
\newcommand{\te}[0]{{\ensuremath{\widetilde{e}}}}
\newcommand{\trho}[0]{{\ensuremath{\widetilde{\rho}}}}
\newcommand{\tsigma}[0]{{\ensuremath{\widetilde{\sigma}}}}
\newcommand{\tomega}[0]{{\ensuremath{\widetilde{\omega}}}}
\newcommand{\tm}[0]{{\ensuremath{\widetilde{m}}}}
\newcommand{\fX}[0]{{\ensuremath{\mathfrak{X}}}}

\newcommand{\fract}[0]{\ensuremath{\operatorname{Frac}}}

\newcommand{\spec}[0]{\ensuremath{\operatorname{Spec}}}

\newcommand{\supp}[0]{\ensuremath{\operatorname{Supp}}}

\newcommand{\Aut}[0]{\ensuremath{\operatorname{Aut}}}

\newcommand{\divi}[0]{\ensuremath{\operatorname{div}}}

\newcommand{\pol}[0]{\ensuremath{\operatorname{Pol}}}

\newcommand{\rank}[0]{\ensuremath{\operatorname{rank}}}
\newcommand{\cone}[0]{\ensuremath{\operatorname{cone}}}
\newcommand{\homo}[0]{\ensuremath{\operatorname{Hom}}}
\newcommand{\relint}[0]{\ensuremath{\operatorname{rel.int}}}
\newcommand{\conv}[0]{\ensuremath{\operatorname{Conv}}}
\newcommand{\trdeg}[0]{\ensuremath{\operatorname{tr.deg}}}

\newcommand{\SL}[0]{\ensuremath{\operatorname{SL}_2}}
\newcommand{\slt}[0]{\ensuremath{\mathfrak{sl}_2}}
\newcommand{\PSL}[0]{\ensuremath{\operatorname{PSL}_2}}
\newcommand{\PGL}[0]{\ensuremath{\operatorname{PGL}}}
\newcommand{\Der}[0]{\ensuremath{\operatorname{Der}}}
\newcommand{\mfs}[0]{\ensuremath{\mathfrak{s}}}

\begin{document}

%\doublespacing

\title[P-divisors and $\SL$-actions on $\TT$-varieties]{Polyhedral
  divisors and $\mathbf{SL_{\mathbf{2}}}$-actions \\ on affine $\TT$-varieties}

\author{Ivan Arzhantsev and Alvaro Liendo}

\address{Department of Higher Algebra, Faculty of Mechanics and
  Mathematics, Lomonosov Moscow State University, Leninskie Gory 1,
  Moscow 119991, Russia} %
\email{arjantse@mccme.ru}

\address{Mathematisches Institut, Universit\"at Basel, Rheinsprung 21,
  CH-4051 Basel, Switzerland}%
\email{alvaro.liendo@gmail.com}

\date{\today}

\thanks{\mbox{\hspace{11pt}}{\it 2010 Mathematics Subject
    Classification}:
  13N15, 14L30, 14M25, 14R20, 14M17.\\
  \mbox{\hspace{11pt}}{\it Key words}: $\SL$-actions, torus actions,
  affine varieties, locally nilpotent derivations, special actions,
  quasi-homogeneous $\SL$-threefolds. \\
  \mbox{\hspace{11pt}} The first author was partially supported by RFBR grant 09-01-00648-a
  and the Simons Foundation.}

\begin{abstract}
  In this paper we classify $\SL$-actions on normal affine
  $\TT$-varieties that are normalized by the torus $\TT$. This is done
  in terms of a combinatorial description of $\TT$-varieties given by
  Altmann and Hausen. The main ingredient is a generalization of
  Demazure's roots of the fan of a toric variety. As an application we
  give a description of special $\SL$-actions on normal affine
  varieties. We also obtain, in our terms, the classification of
  quasihomogeneous $\SL$-threefolds due to Popov.
\end{abstract}

\maketitle

\section*{Introduction}

Let $\KK$ be an algebraically closed field of characteristic zero, $M$ be
a lattice of rank $n$, $N=\homo(M,\ZZ)$ be the dual lattice of $M$,
and $\TT$ be the algebraic torus $\spec\KK[M]$, so that $M$ is the
character lattice of $\TT$ and $N$ is the one-parameter subgroup
lattice of $\TT$.

A \emph{$\TT$-variety} $X$ is a normal algebraic variety endowed with
an effective regular action of $\TT$. The \emph{complexity} of a
$\TT$-action is the codimension of a general orbit, and since the
$\TT$-action on $X$ is effective, the complexity of $X$ equals $\dim
X-\rank M$. For an affine variety $X$, to introduce a $\TT$-action on
$X$ is the same as to endow $\KK[X]$ with an $M$-grading. There are
well known combinatorial descriptions of $\TT$-varieties. We send the
reader to \cite{Dem70} and \cite{Fu93} for the case of toric
varieties, to \cite[Ch. 2 and 4]{KKMS73} and \cite{Tim08} for the
complexity one case, and to \cite{AlHa06,AHS08} for the general
case. In this paper we use the approach in \cite{AlHa06}.

Any affine toric variety is completely determined by a polyhedral cone
$\sigma\subseteq N_{\QQ}$. Similarly, the description of a normal
affine $\TT$-varieties $X$ due to Altmann and Hausen \cite{AlHa06}
involves the data $(Y,\sigma,\DD)$ where $Y$ is a normal
semiprojective variety, $\sigma\subseteq N_{\QQ}:=N\otimes\QQ$ is a
polyhedral cone, and $\DD$ is a divisor on $Y$ whose coefficients are
polyhedra in $N_{\QQ}$ with tail cone $\sigma$. The divisor $\DD$ is
called a \emph{$\sigma$-polyhedral divisor} on $Y$ (see
Section~\ref{comb-des} for details).

Let $X$ be a $\TT$-variety endowed with a regular $G$-action, where
$G$ is any linear algebraic group. We say that the $G$-action on $X$
is \emph{compatible} if the image of $G$ in $\Aut(X)$ is normalized
but not centralized by $\TT$. Furthermore, we say that the $G$-action
is of \emph{fiber type} if the general orbits are contained in the
$\TT$-orbit closures, and of \emph{horizontal type} otherwise
\cite{FlZa05b,Lie10}.

Let now $\GA=\GA(\KK)$ be the additive group of $\KK$. It is well
known that a $\GA$-action on an affine variety $X$ is equivalent to a
locally nilpotent derivation (LND) of $\KK[X]$. A description of
compatible $\GA$-actions on an affine $\TT$-variety, or equivalently
of homogeneous LNDs on $\KK[X]$, is available in the case where $X$ is
of complexity at most one \cite{Lie10} or the $\GA$-action is of fiber
type \cite{Lie10b} in terms of a generalization of Demazure's roots of
a fan \cite{Dem70} (see Sections~\ref{LND-toric} and \ref{sec:LND-T}).

A regular $\SL$-action on an affine variety $X$ is uniquely defined by
an $\slt$-triple $\{\delta,\partial_+,\partial_-\}$ of derivations of
the algebra $\KK[X]$, where $\partial_\pm$ are locally nilpotent,
$\delta=[\partial_+,\partial_-]$ is semisimple and
$[\delta,\partial_\pm]=\pm2\partial_\pm$ (see
Proposition~\ref{sl-actions}). Assume now that $X$ is an affine
$\TT$-variety. If the $\SL$-action is compatible, then $\partial_\pm$
are homogeneous with respect to the $M$-grading on $\KK[X]$ and the
grading given by $\delta$ is a downgrading of the $M$-grading.

The main result of this paper, contained in Section~\ref{sec:main}, is
a classification of compatible $\SL$-actions on an affine
$\TT$-variety $X$ in the case where this action is of fiber type or
$X$ is of complexity one (See Theorems~\ref{sec:sl-fiber} and
\ref{th:hor}, respectively).  Our idea is to classify compatible
$\SL$-actions by calculating the commutator of two homogeneous
LNDs. The existence of a compatible $\SL$-action on $X$ puts strong
restrictions on the combinatorial data $(Y,\sigma,\DD)$ and endows
$\DD$ with an additional structure. It should be noted that if the
$\TT$-variety $X$ is of complexity one and the $\SL$-action is of
horizontal type, then $X$ is spherical with respect to a bigger
reductive group, namely, an extension of $\SL$ by a torus. We do not
use the theory of spherical varieties in this paper.

The rest of the paper is devoted to two applications of our main
result: special $\SL$-actions and $\SL$-actions with an open orbit. A
$G$-action on $X$ is called special (or horospherical) if there exists
a dense open $W\subseteq X$ such that the isotropy group of any point
$x\in W$ contains a maximal unipotent subgroup of $G$. Special actions
play an important role in Invariant Theory.

Any special action of a connected reductive group $G$ on an affine
variety $X$ may be reconstructed from the action of a maximal torus
$T\subseteq G$ on the algebra $\KK[X]^U$ of invariants of a maximal
unipotent subgroup $U$ \cite[Theorem~5]{Pop86}. This reduces the study
of special actions to torus actions. In Section~\ref{sec:special} we
illustrate this phenomenon for $\SL$-actions in our terms (see
Theorem~\ref{th:special} and \ref{rk:special-easy}). In particular, we
show that for every special $\SL$-action on an affine variety $X$
there is a canonical 2-torus action and the $\SL$-action is compatible
and of fiber type with respect to this torus. Since the reconstruction
of the $G$-variety $X$ from the $T$-variety $\spec\KK[X]^U$ is an
algebraic procedure, it is useful to have a geometric description of
$X$. In Proposition~\ref{lm:T2-ex} we describe a normal affine variety
$X$ with a special $\SL$-action as a $\TT^2$-variety with respect to
the canonical torus $\TT^2$. It is worthwhile to remark that any
$G$-action on an affine variety may be contracted to a special one
\cite[Proposition~8]{Pop86}. It will be interesting to interpret
contraction of $\SL$-actions in terms of polyhedral divisors.

As a corollary of our classification of special actions, we prove that
if an affine $\TT$-variety $X$ of complexity one admits a compatible
special $\SL$-action of horizontal type, then $X$ is toric with
respect to a bigger torus and the $\SL$-action is compatible with
respect to the big torus as well. Furthermore, using a linearization
result due to Berchtold and Hausen \cite{BeHa03}, we show that, up to
conjugation in $\Aut(X)$, any special $\SL$-action on an affine toric
3-fold $X$ is compatible with the big torus, and thus it is given by
an $\SL$-root (see Definition~\ref{def:sl-root}).

It is natural to generalize Altmann and Hausen's approach
\cite{AlHa06} to arbitrary reductive groups. Special actions form the
most accessible class for such a generalization. Our work in this line
may be regarded as a first step towards this aim. It must be noted
that Timashev \cite{Tim97} already gave a combinatorial description
for $G$-actions of complexity one in the framework of Luna-Vust
theory.

Finally, our method allows to reprove, in Section~\ref{sec:popov},
Popov's classification of generically transitive $\SL$-actions on
normal affine threefolds. The only fact that we use is the existence
of a one dimensional torus $R$ commuting with $\SL$. Together with the
maximal torus in $\SL$, this allows us to consider a quasi-homogeneous
threefold as a $\TT^2$-variety of complexity one, where $\TT^2$ is a
two dimensional torus. We also obtain, as a direct consequence of our
results, the characterization of toric quasi-homogeneous
$\SL$-threefolds given in \cite{Gai08} and \cite{BaHa08} (see
Corollaries~\ref{cor:gaif} and \ref{cor:gaif-2}). Recall that a
$G$-variety is quasi-homogeneous if it has an open $G$-orbit.

In the entire paper the term variety means a normal integral scheme of
finite type over an algebraically closed field $\KK$ of characteristic
zero. The term point always refer to a closed point.

\smallskip 

The authors are grateful to Mikhail Zaidenberg for useful discussions
and the referee for valuable suggestions. This work was done during
stays of both authors at the Institut Fourier, Grenoble. We thank the
Institut Fourier for its support and hospitality.

\tableofcontents

\section{Preliminaries}
\label{pre}

In this section we recall the results about $\GA$-actions on
affine $\TT$-varieties needed in this paper.

\subsection{Combinatorial description of $\TT$-varieties}
\label{comb-des}

Let $M$ be a lattice of rank $n$ and $N=\homo(M,\ZZ)$ be its dual
lattice. We let $M_{\QQ}=M\otimes\QQ$, $N_{\QQ}=N\otimes\QQ$, and we
consider the natural duality pairing $M_{\QQ}\times N_{\QQ}\rightarrow
\QQ$, $(m,p)\mapsto \langle m,p\rangle=p(m)$.

Let $\TT=\spec\KK[M]$ be the $n$-dimensional algebraic torus
associated to $M$ and let $X=\spec\,A$ be an affine $\TT$-variety. The
comorphism $A\rightarrow A\otimes \KK[M]$ induces an $M$-grading on
$A$ and, conversely, every $M$-grading on $A$ arises in this way. The
$\TT$-action on $X$ is effective if and only if the corresponding
$M$-grading is effective.

In \cite{AlHa06}, a combinatorial description of normal affine
$\TT$-varieties is given. In what follows we recall the main features
of this description. Let $\sigma$ be a pointed polyhedral cone in
$N_{\QQ}$. We define $\pol_{\sigma}(N_{\QQ})$ to be the set of all
\emph{$\sigma$-polyhedra} i.e., the set of all polyhedra in $N_{\QQ}$
that can be decomposed as the Minkowski sum of a bounded polyhedron
and the cone $\sigma$.

Recall that $\sigma^\vee$ stands for the cone in $M_\QQ$ dual to
$\sigma$. To a $\sigma$-polyhedron $\Delta\in\pol_{\sigma}(N_{\QQ})$
we associate its support function $h_{\Delta}:\sigma^\vee\rightarrow
\QQ$ defined by
$$h_{\Delta}(m)=\min\langle m,\Delta\rangle=\min_{p\in \Delta}\langle
m,p\rangle\,.$$ %
Furthermore, if we let $\{v_i\}$ be the set of all vertices of
$\Delta$, then the support function is given by
\begin{align} \label{support-function}
h_{\Delta}(m)=\min_i\{v_i(m)\}\quad\mbox{for all}\quad
m\in\sigma^\vee\,.
\end{align}
Hence, $h_\Delta$ is piecewise linear, concave, and positively
homogeneous.

\begin{definition} \label{ppd} A normal variety $Y$ is called
  \emph{semiprojective} if it is projective over an affine variety. A
  \emph{$\sigma$-polyhedral divisor} on $Y$ is a formal sum
  $\DD=\sum_{Z}\Delta_Z\cdot Z$, where $Z$ runs over all prime
  divisors on $Y$, $\Delta_Z\in\pol_{\sigma}(N_{\QQ})$, and
  $\Delta_Z=\sigma$ for all but finitely many $Z$. For
  $m\in\sigma^{\vee}$ we can evaluate $\DD$ in $m$ by letting $\DD(m)$
  be the $\QQ$-divisor
  $$\DD(m)=\sum_{Z\subseteq Y} h_Z(m)\cdot Z\,,$$
  where $h_Z$ is the support function of $\Delta_Z$. A
  $\sigma$-polyhedral divisor $\DD$ is called \emph{proper} if the
  following hold:
\begin{enumerate}[(i)]
\item $\DD(m)$ is semiample and $\QQ$-Cartier for all
  $m\in\sigma^\vee$, and
\item $\DD(m)$ is big for all $m\in\relint(\sigma^\vee)$.
\end{enumerate}
\end{definition}

Here $\relint(\sigma^\vee)$ denotes the relative interior of the cone
$\sigma^\vee$. Furthermore, a $\QQ$-Cartier divisor $D$ on $Y$ is
called \emph{semiample}\index{semiample divisor} if there exists $r>0$
such that the linear system $|rD|$ is base point free, and
\emph{big}\index{big divisor} if there exists a divisor $D_0\in |rD|$,
for some $r>0$, such that the complement $Y\setminus \supp D_0$ is
affine.

The following theorem gives a combinatorial description of
$\TT$-varieties analogous to the classical combinatorial description
of toric varieties. In the sequel, $\chi^m$ denotes the character of
$\TT$ corresponding to the lattice vector $m$, and $\sigma^\vee_M$
denotes the semigroup $\sigma^\vee\cap M$. Furthermore, for a
$\QQ$-divisor $D$ on $Y$, $\OO_Y(D)$ stands for the sheaf
$\OO_Y(\lfloor D\rfloor)$.

\begin{theorem}[\cite{AlHa06}] \label{AH} To any proper
  $\sigma$-polyhedral divisor $\DD$ on a semiprojective variety $Y$
  one can associate a normal affine $\TT$-variety of dimension $\rank
  M+\dim Y$ given by $X[Y,\DD]=\spec A[Y,\DD]$, where
  $$A[Y,\DD]=\bigoplus_{m\in\sigma^{\vee}_M} A_m\chi^m,\quad
  \mbox{and}\quad A_m=H^0(Y,\OO_Y(\DD(m))\subseteq \KK(Y)\,.$$

  Conversely, any normal affine $\TT$-variety is isomorphic to
  $X[Y,\DD]$ for some semiprojective variety $Y$ and some proper
  $\sigma$-polyhedral divisor $\DD$ on $Y$.
\end{theorem}

We call $Y$ the \emph{base variety} and the pair $(Y,\DD)$ the
\emph{combinatorial description} of $X$. We also define the
\emph{support} of a proper $\sigma$-polyhedral divisor as $\supp
\DD=\{Z\subseteq Y\mid \Delta_Z\neq \sigma\}$.

This combinatorial description is not unique, but
can be made unique by adding some minimality conditions on the pair
$(Y,\DD)$, see \cite[Section 8]{AlHa06}. Here we only need a
particular case of \cite[Corollary 8.12]{AlHa06}.

\begin{corollary} \label{cor:AH} %
  Let $\DD$ and $\DD'$ be two proper $\sigma$-polyhedral divisors on a
  normal semiprojective variety $Y$.  If for every prime divisor $Z$
  in $Y$ there exists a vector $v_Z\in N$ such that
  $$\DD=\DD'+\sum_Z(v_Z+\sigma)\cdot Z,\quad\mbox{and}\quad
  \sum_Z\langle m,v_Z\rangle\cdot Z \mbox{ is principal,}\quad \forall
  m\in\sigma^\vee_M\,,$$ then $X[Y,\DD]$ is equivariantly isomorphic
  to $X[Y,\DD']$.
\end{corollary}

Most of this paper deals with the case where the base is a curve $C$
isomorphic to $\AF^1$ or $\PP^1$. Any $\sigma$-polyhedral divisor on
$\AF^1$ is proper. If $\DD=\sum_{z\in C}\Delta_z\cdot z$ is a
$\sigma$-polyhedral divisor on $C=\PP^1$, then $\DD$ is proper if and
only if $\deg\DD:=\sum_{z\in C}\Delta_z\subsetneq \sigma$. We also
need the following result from \cite[Section 11]{AlHa06}.

\begin{corollary} \label{AH-toric} Let $\DD$ be a proper
  $\sigma$-polyhedral divisor on a smooth curve $C$. Then $X[C,\DD]$
  is toric if and only if $C=\AF^1$ and $\DD$ can be chosen (via
  Corollary~\ref{cor:AH}) supported in at most one point, or $C=\PP^1$
  and $\DD$ can be chosen (via Corollary~\ref{cor:AH}) supported in at
  most two points.
\end{corollary}

\subsection{Locally nilpotent derivations and $\GA$-actions}

Let $X=\spec\,A$ be an affine variety. A~derivation $\partial$ on $A$
is called \emph{locally nilpotent} (LND for short) if for every $a\in
A$ there exists $n\in\ZZ_{\geq 0}$ such that $\partial^n(a)=0$. We
denote by $\GA$ the additive group of the base field $\KK$. Given an
LND $\partial$ on $A$, the map $\phi_\partial:\GA\times A\rightarrow
A$, $\phi_\partial(t,f)=\exp(t\partial)(f)$ defines a $\GA$-action on
$X$, and any $\GA$-action on $X$ arises in this way \cite{Fre06}.

Let now $\DD$ be a proper $\sigma$-polyhedral divisor on a
semiprojective variety $Y$, and let $A=A[Y,\DD]$ be the corresponding
$M$-graded domain. A $\GA$-action on $X=\spec A$ is said
\emph{compatible} with the $\TT$-action on $X$ if the image of $\GA$
in $\Aut(X)$ is normalized by the torus $\TT$. A $\GA$-action is
compatible if and only if the corresponding LND $\partial$ on $A$ is
homogeneous i.e., if $\partial$ sends homogeneous elements to
homogeneous elements. Any homogeneous LND $\partial$ has a well
defined degree given as $\deg\partial=\deg \partial(f)-\deg f$ for any
homogeneous $f\in A\setminus\ker\partial$.

A homogeneous LND $\partial$ on $A$ extends to a derivation on $\fract
A=\KK(Y)(M)$ by the Leibniz rule, where $\KK(Y)(M)$ is the field of
fractions of $\KK(Y)[M]$. The LND $\partial$ is said to be \emph{of
  fiber type} if $\partial(\KK(Y))=0$ and \emph{of horizontal type}
otherwise. Geometrically speaking, $\partial$ is of fiber type if and
only if the general orbits of the corresponding $\GA$-action on
$X=\spec\,A$ are contained in the orbit closures of the $\TT$-action
given by the $M$-grading.

\subsection{Locally nilpotent derivations on affine toric
  varieties} \label{LND-toric}

In this section we recall the classification of homogeneous LNDs on
toric varieties given in \cite{Lie10}. A similar description is
implicit in \cite[Section 4.5]{Dem70}. As usual, for a cone $\sigma$,
we denote by $\sigma(1)$ the set of all rays of~$\sigma$ and we
identify a ray with its primitive vector.

\begin{definition} \label{srho} %
  Let $\sigma$ be a pointed cone in $N_\QQ$. We say that $e\in M$ is a
  \emph{root} of the cone $\sigma$ if the following hold:
  \begin{enumerate}[$(i)$]
  \item there exists $\rho_e\in \sigma(1)$ such that
    $\langle e,\rho_e\rangle=-1$;
  \item $\langle e,\rho\rangle\geq 0$, for all
    $\rho\in\sigma(1)\setminus\{\rho_e\}$.
\end{enumerate}
The ray $\rho_e$ is called the \emph{distinguished} ray of the root
$e$. We denote by $\RR(\sigma)$ the set of all roots of $\sigma$.
\end{definition}

One easily checks that any ray $\rho\in\sigma(1)$ is the distinguished
ray for infinitely many roots $e\in\RR(\sigma)$.  For every root $e\in
\RR(\sigma)$ we define a homogeneous derivation $\partial_e$ of degree
$e$ of the algebra $\KK[\sigma^\vee_M]$ by the formula
$$\partial_e(\chi^m)=\langle m,\rho_e\rangle\cdot
\chi^{m+e},\quad\mbox{for all}\quad m\in \sigma^\vee_M\,.$$

The following theorem gives a classification of the homogeneous LNDs
on $\KK[\sigma^\vee_M]$.

\begin{theorem} \label{toric} %
  For every root $e\in \RR(\sigma)$, the homogeneous derivation
  $\partial_{e}$ on $\KK[\sigma_M^\vee]$ is an LND of degree $e$ with
  kernel $\ker\partial_{e}=\KK[\tau_e\cap M]$, where $\tau_e$ is the facet of
  $\sigma^\vee$ dual to the distinguished ray $\rho_e$. Conversely, if
  $\partial\neq 0$ is a homogeneous LND on $\KK[\sigma_M^\vee]$, then
  $\partial=\lambda\partial_{e}$ for some root $e\in \RR(\sigma)$, and
  some $\lambda\in\KK^*$.
\end{theorem}

\subsection{Locally nilpotent derivations on affine $\TT$-varieties}
\label{sec:LND-T}

We give first a classification of homogeneous LNDs of fiber type on
$\TT$-varieties of arbitrary complexity given in \cite{Lie10b}.

Letting $\DD=\sum_Z\Delta_Z\cdot Z$ be a proper $\sigma$-polyhedral
divisor on a semiprojective variety $Y$, we let $A=A[Y,\DD]$. For
every prime divisor $Z\subseteq Y$, we let $\{v_{i,Z}\mid
i=1,\cdots,r_Z\}$ be the set of all vertices of $\Delta_Z$. Letting
$e$ be a root of the cone $\sigma$, we define the divisor
$$\DD(e)= \sum_Z \min_i\{v_{i,Z}(e)\}\cdot Z,\quad\mbox{and}\quad \Phi_e^*=
H^0(Y,\OO_Y(\DD(e)))\setminus \{0\}\,.$$ %
Remark that the evaluation divisor $\DD(m)$ is only defined for $m\in
\sigma^\vee$ and $e\notin\sigma^\vee$. The reason for the above
notation is that taking \eqref{support-function} as the definition of
support function, we obtain the above formula for the evaluation
divisor, which can be evaluated at any $m\in M_\QQ$.

For every $\varphi\in\Phi_e^*$ we let
$$\partial_{e,\varphi}(f\chi^m)=\langle m,\rho_e\rangle\cdot
\varphi\cdot f\chi^{m+e},\quad\mbox{for all}\quad m\in
\sigma^\vee_M,\quad\mbox{and}\quad f\in \KK(Y)\,.$$

The following theorem gives a classification of the homogeneous LNDs
of fiber type on $A[Y,\DD]$.

\begin{theorem} \label{lnd-fiber} %
  For every root $e\in\RR(\sigma)$ and $\varphi\in\Phi_e^*$, the
  derivation $\partial_{e,\varphi}$ is a homogeneous LND of fiber type
  on $A=A[Y,\DD]$ of degree $e$ with kernel
  $$\ker\partial_{e,\varphi}=\bigoplus_{m\in\tau_e\cap M}A_m\chi^m\,,$$
  where $\tau_e\subseteq \sigma^\vee$ is the facet dual to the
  distinguished ray $\rho_e$. Conversely, if $\partial\neq 0$ is a
  homogeneous LND of fiber type on $A$, then
  $\partial=\partial_{e,\varphi}$ for some root $e\in \RR(\sigma)$ and
  some $\varphi\in\Phi_e^*$.
\end{theorem}

\medskip

The classification of LNDs of horizontal type is more involved and is
only available in the case of complexity one. Here, we give an
improved presentation of the classification given in \cite[Theorem
3.28]{Lie10}.

Since the complexity is one the base variety $Y$ is a smooth curve
$C$. Let $\DD=\sum_{z\in C}\Delta_z\cdot z$ be a proper
$\sigma$-polyhedral divisor on $C$, and let $X=X[C,\DD]$.  If
$A=A[C,\DD]$ admits a homogeneous LND of horizontal type, then $C$ is
isomorphic either to $\AF^1$ or to $\PP^1$. In the following we assume
that $C=\AF^1$ or $C=\PP^1$.

\begin{definition} \label{def:col-pol}
  A \emph{colored} $\sigma$-polyhedral divisor on $C$ is a collection
  $\tDD=\{\DD;v_z,\forall z\in C\}$ if $C=\AF^1$ and
  $\tDD=\{\DD,z_\infty;v_z,\forall z\in C\setminus{z_\infty}\}$ if
  $C=\PP^1$, satisfying the following conditions:
  \begin{enumerate}[(1)]
  \item $\DD=\sum_{z\in C}\Delta_z\cdot z$ is a proper
    $\sigma$-polyhedral divisor on $C$, $z_\infty\in C$, and $v_z$ is
    a vertex of $\Delta_z$. We let $C'=C$ if $C=\AF^1$ and
    $C'=C\setminus \{z_\infty\}$ if $C=\PP^1$.
  \item $v_{\deg}:=\sum_{z\in C'} v_z$ is a vertex of $\deg \DD|_{C'}$, and
  \item $v_z\in N$ with at most one exception.
  \end{enumerate}

  We also let $z_0\in C'$ be such that $v_z\in N$ for all $z\in
  C'\setminus\{z_0\}$. We say that $\tDD$ is a \emph{coloring} of
  $\DD$ and we call $z_0$ the \emph{marked point}, $z_\infty$ the
  \emph{point at infinity} if $C=\PP^1$, and $v_z$ the \emph{colored
    vertex} of the polyhedron $\Delta_z$. 
\end{definition}

Remark that the above notion of coloring is independent from the
notion of coloring in the theory of spherical varieties.

Let $\tDD$ be a colored $\sigma$-polyhedral divisor on $C$. Letting
$\omega\subseteq N_\QQ$ be the cone generated by $\deg\DD|_{C'}-v_{\deg}$ we
let $\tomega\subseteq (N\oplus\ZZ)_\QQ$ be the cone generated by
$(\omega,0)$ and $(v_{z_0},1)$ if $C=\AF^1$, and by $(\omega,0)$,
$(v_{z_0},1)$ and $(\Delta_{z_\infty}+v_{\deg}-v_{z_0}+\omega,-1)$ if
$C=\PP^1$. Denote by $d$ the minimal positive integer such
that $d\cdot v_{z_0}\in N$. We call $\tomega$ the \emph{associated
  cone} of the colored $\sigma$-polyhedral divisor~$\DD$.

\begin{definition} \label{col-div} %
  A pair $(\tDD,e)$, where $\tDD$ is a colored $\sigma$-polyhedral
  divisor on $C$ and $e\in M$, is said to be \emph{coherent} if
  \begin{enumerate} [$(1)$]
  \item There exists $s\in\ZZ$ such that $\te=(e,s)\in M\oplus\ZZ$ is
    a root of the associated cone $\tomega$ with distinguished ray
    $\trho=(d\cdot v_{z_0},d)$. In this case
    $s=-\nicefrac{1}{d}-v_{z_0}(e)$.
  \item $v(e)\geq 1+v_{z}(e)$, for every $z\in C'\setminus\{z_0\}$
    and every vertex $v\neq v_z$ of the polyhedron $\Delta_z$.
  \item $d\cdot v(e)\geq 1+d\cdot v_{z_0}(e)$, for every vertex $v\neq
    v_{z_0}$ of the polyhedron $\Delta_{z_0}$.
  \item If $Y=\PP^1$, then $d\cdot v(e)\geq -1-d\cdot v_{\deg}(e)$, for every vertex $v$ of the polyhedron
    $\Delta_{z_\infty}$.
  \end{enumerate}
\end{definition}

Let now $L=\{m\in M\mid v_{z_0}(m)\in \ZZ\}$ and $\varphi^m\in \KK(C)$
be a rational function with
$$\divi(\varphi^m)|_{C'}+\DD(m)|_{C'}=0,\quad\mbox{and}\quad
\varphi^m\cdot\varphi^{m'}=\varphi^{m+m'} \quad\mbox{for all}\quad
m,m'\in \omega^\vee_L\,.$$ %
The choice of $\varphi^m$ as above is possible since $\DD(m)$ is
linear for $m\in \omega^\vee$.

The following theorem gives a classification of homogeneous LNDs of
horizontal type on $A[C,\DD]$. It corresponds to \cite[Theorem
3.28]{Lie10}.

\begin{theorem} \label{lnd-hor} %
  Let $X=X[C,\DD]$ be a normal affine $\TT$-variety of complexity
  one. Then the homogeneous LNDs of horizontal type on
  $\KK[X]=A[C,\DD]$ are in bijection with the coherent pairs
  $(\tDD,e)$, where $\tDD$ is a coloring of $\DD$ and $e\in
  M$. Furthermore, the homogeneous LND $\partial$ corresponding to
  $(\tDD,e)$ has degree $e$ and kernel
  $$\ker\partial=\bigoplus_{m\in \omega^\vee_L} \KK\varphi^m\,.$$
\end{theorem}

Let us give an explicit formula for the homogeneous LND $\partial$
associated to the coherent pair $(\tDD,e)$. Without loss of generality
we may assume $z_0=0$ and $z_\infty=\infty$ if $C=\PP^1$. By
Corollary~\ref{cor:AH} we may assume $v_{z}=\bar{0}\in N$ for all
$z\in C'\setminus\{z_0\}$. Letting $\KK[C']=\KK[t]$, the homogeneous
LND of horizontal type $\partial$ corresponding to the coherent pair
$(\tDD,e)$ is given by
\begin{align}
  \label{eq:1}
  \partial(\chi^m\cdot t^r)=d(v_{0}(m)+r)\cdot\chi^{m+e}\cdot
  t^{r+s}, \quad\mbox{for all}\quad (m,r)\in
  M\oplus \ZZ\,.
\end{align}
Furthermore, if we let $\tm=(m,r)\in M\oplus \ZZ$ and 
$\chi^{\tm}=\chi^m\cdot t^r$, then \eqref{eq:1} can be written
as in the toric case
\begin{align}
  \label{eq:2}
  \partial(\chi^{\tm})=\langle\tm,\trho\rangle \cdot\chi^{\tm+\te},
  \quad\mbox{for all}\quad \tm\in M\oplus \ZZ\,.
\end{align}

We also need the following two technical lemmas. They follow from
Theorem~\ref{lnd-hor}. 

\begin{lemma}[cf. {\cite[Lemma~4.5]{Lie10}}] \label{lm:kernel} Let
  $X=X[C,\DD]$ and let $\partial_1$ and $\partial_2$ be two
  homogeneous LNDs of horizontal type on $\KK[X]$. Assume that
  $z_\infty(\partial_1)=z_\infty(\partial_2)$. Then
  $\ker\partial_1\cap\ker\partial_2\supsetneq \KK$ if and only if
  $\omega(\partial_1)\cap\omega(\partial_2)\supsetneq
  \{0\}$. Furthermore, if $\rank M=2$ this is the case if and only if
  the vertices $v_{\deg}(\partial_1)$ and $v_{\deg}(\partial_1)$ are
  adjacent vertices in the polyhedron $\deg \DD|_{C'}$.
\end{lemma}

\begin{lemma}[cf. {\cite[Remark~3.27]{Lie10}}] \label{hor-int} %
  Let $X=X[C,\DD]$ and let $\partial$ be a homogeneous LND of
  horizontal type on $\KK[X]$ of degree $e$. Then
  \begin{enumerate}
  \item If $C=\AF^1$ then $e\in \omega^\vee\subseteq\sigma^\vee$.
  \item If $C=\PP^1$ and for every ray $\rho\in
    \sigma(1)\cap\omega(1)$ we have $\rho\cap \deg\DD=\emptyset$, then
    $e\in \omega^\vee\subseteq\sigma^\vee$.
  \end{enumerate}

\end{lemma}

\section{Compatible $\SL$-actions on normal affine $\TT$-varieties}
\label{sec:main}

In this section we give a classification of compatible $\SL$-actions
on $\TT$-varieties in two cases: in the case where the $\TT$-action is
of complexity one; and in arbitrary complexity provided that the
general $\SL$-orbits are contained in the $\TT$-orbit closures.

\subsection{$\SL$-actions on affine varieties}

Let $\SL$ be the algebraic group of $2\times 2$ matrices of
determinant 1. Every algebraic subgroup of $\SL$ of positive dimension
is conjugate to one of the following subgroups:
$$T=\left\{
  \left(\begin{smallmatrix}
    t & 0 \\
    0 & t^{-1}
  \end{smallmatrix}\right)
  \mid t\in\KK^*\right\},\quad 
  U_{(s)}=\left\{
  \left(\begin{smallmatrix}
    \epsilon & \lambda \\
    0 & \epsilon^{-1}, 
  \end{smallmatrix}\right)
\mid \epsilon,\lambda\in\KK, \epsilon^s=1\right\},$$
$$N=T\cup \left(
  \begin{smallmatrix}
    0 & 1 \\
    -1 & 0
  \end{smallmatrix}\right)\cdot T,
\quad \mbox{and} \quad B=T\cdot U_{(1)}\,.$$

Here $T$ is a maximal torus, $N$ is the normalizer of a maximal torus,
$B$ is a Borel subgroup, and $U_{(s)}$ is a cyclic extension of a
maximal unipotent subgroup. We also define the following 
maximal unipotent subgroups:
$$U_+=U_{(1)}, \quad\mbox{and}\quad U_-=
\left(
  \begin{smallmatrix}
    0 & -1 \\
    1 & 0
  \end{smallmatrix}\right)\cdot U_{(1)}\cdot
\left(
  \begin{smallmatrix}
    0 & 1 \\
    -1 & 0
  \end{smallmatrix}\right)\,.
$$
As a group $\SL$ is generated by the unipotent subgroups $U_+$ and
$U_-$ isomorphic to $\GA$.

Let now $X=\spec A$ be an affine variety endowed with an
$\SL$-action. The two $U_\pm$-actions on $X$ are equivalent to two
LNDs $\partial_\pm$ on the algebra $A$, and the $T$-action on $X$ is
equivalent to a $\ZZ$-grading on $A$. Furthermore, this $\ZZ$-grading
on $A$ is also uniquely determined by its infinitesimal generator
i.e., by the semisimple derivation $\delta$ given by
$\delta(a)=\deg(a)\cdot a$ for every homogeneous $a\in A$.

The following well known proposition gives a criterion for the
existence of an $\SL$-action on an affine variety. In the lack of a
reference we provide a short proof, cf. {\cite[4.15]{FlZa05b}}.

\begin{proposition} \label{sl-actions} %
  A non-trivial $\SL$-action on an affine variety $X=\spec A$ is
  equivalent to a (not necessarily effective) $\ZZ$-grading on $A$
  with infinitesimal generator $\delta$ and a couple of homogeneous
  LNDs $(\partial_+,\partial_-)$ of degrees $\deg_\ZZ\partial_\pm=\pm
  2$, satisfying $[\partial_+,\partial_-]=\delta$.

  Furthermore, the $\ZZ$-grading is effective if and only if $\SL$
  acts effectively on $X$. If the $\ZZ$-grading is not effective,
  then the kernel of $\SL\rightarrow \Aut(X)$ is $\{\pm
  \operatorname{Id}\}$ and so $\PSL$ acts effectively on $X$.
\end{proposition}
\begin{proof}
  Assume first that $\SL$ acts non-trivially on $X$. Let
  $\{h,e_+,e_-\}$ be the $\slt$-triple in the Lie algebra
  $\slt$. Since $e_\pm$ is tangent to the 1-parameter unipotent
  subgroup $U_\pm$ in $\SL$, it acts on $A$ as an LND
  $\partial_\pm$. The vector $h$ is tangent to the torus $T$, thus $h$
  acts on $A$ as the infinitesimal generator $\delta$ of a
  $\ZZ$-grading. Since $[h,e_\pm]=\pm 2e_\pm$, the LND $\partial_\pm$
  is homogeneous of degree $\pm 2$, and the relation $[e_+,e_-]=h$
  implies $[\partial_+,\partial_-]=\delta$.

  Conversely, assume that we have $\delta,\partial_+,\partial_-$ as in
  the proposition. Then $\mfs=\langle
  \delta,\partial_+,\partial_-\rangle$ is a Lie subalgebra in
  $\Der(A)$ isomorphic to $\slt$. Furthermore, every element of $A$ is
  contained in a finite-dimensional $\mfs$-submodule. Recall that any
  finite-dimensional $\slt$-module has a canonical structure of
  $\SL$-module whose tangent representation coincides with the given
  one. This gives $A$ the structure of a rational $\SL$-module. Since
  $\SL$ is generated as a group by the subgroups $U_\pm$ and the LNDs
  $\partial_\pm$ define the action of $U_\pm$ via automorphisms, the
  group $\SL$ acts on $A$ via automorphisms. This proves that $A$ is a
  rational $\SL$-algebra, or, equivalently, $\SL$ acts regularly on
  $X$.
\end{proof}

We restrict now to the case of affine $\TT$-varieties. The following
definition determines the class of $\SL$-actions that we study in the
sequel.

\begin{definition}
  An $\SL$-action on a $\TT$-variety $X$ is \emph{compatible} if the
  image of $\SL$ in $\Aut(X)$ is normalized but not centralized by the
  torus $\TT$.
\end{definition}

Assume now that $X$ is a $\TT$-variety endowed with a compatible
$\SL$-action. Denote by $\widetilde{\SL}$ the image of $\SL$ in
$\Aut(X)$. There is a homomorphism $\psi:\TT\rightarrow
\Aut(\widetilde{\SL})$. Since any automorphism from
$\Aut(\widetilde{\SL})$ is inner, we have $\Aut(\widetilde{\SL})\simeq
\PSL$. Thus the image of $\TT$ is either trivial or is a maximal torus
$T\subseteq\PSL$. In the first case $\TT$ centralizes
$\widetilde{\SL}$ so this case is excluded by the definition of a
compatible $\SL$-action. Hence $\TT$ contains $T$ and so $\TT=T\cdot
S$, where $S=\ker\psi$ is a complementary subtorus that centralizes
the $\SL$-action. Let $U_\pm$ be unipotent root subgroups in
$\widetilde{\SL}$ with respect to the torus $T$. Then the $\SL$-action
on $X$ is defined by the infinitesimal generator corresponding to a
$\ZZ$-grading on $\KK[X]$ defined by $T$ (this is a downgrading of the
$M$-grading) and two $M$-homogeneous LNDs $\partial_\pm$ corresponding
to the $U_\pm$-actions. This gives the following corollary.

\begin{corollary}\label{cor-downg}
  \begin{enumerate}[$(i)$] 
  \item Let $X$ be a normal affine $\TT$-variety endowed with a
    compatible $\SL$-action. In this case, in
    Proposition~\ref{sl-actions}, we may assume that $\delta$ is the
    infinitesimal generator corresponding to a downgrading of $M$ and
    that $\partial_\pm$ are $M$-homogeneous LNDs. Furthermore,
    $\TT=T\cdot S$, where $T$ is the maximal torus in $\SL$ and $S$ is
    a complementary subtorus that centralizes the $\SL$-action.

  \item Let $X$ be a normal affine $\TT$-variety endowed with an
    $\SL$-action that is centralized by $\TT$. Then we may extend
    $\TT$ by $T$ so that the $\SL$-action is compatible with this
    bigger torus action.
  \end{enumerate}
\end{corollary}

The following is a generalization of a definition in \cite{Lie10b}.

\begin{definition}
  We say that a compatible $\SL$-action on a $\TT$-variety is of
  \emph{fiber type} if the general orbits are contained in the
  $\TT$-orbit closures and of \emph{horizontal type} otherwise.
\end{definition}

Clearly, a compatible $\SL$-action is of fiber type if and only if
both derivations $\partial_\pm$ are of fiber type. The following lemma
shows that a compatible $\SL$-action is of horizontal type if and only
if both derivations $\partial_\pm$ are of horizontal type.

\begin{lemma}\label{lm-unmix}
  Consider a compatible $\SL$-action on a $\TT$-variety $X$ and assume
  that the LND $\partial_+$ is of fiber type. Then the $\SL$-action is
  of fiber type.
\end{lemma}
\begin{proof}
  Set $B=T\cdot U_+\subseteq \SL$. Then the $B$-action on $X$ is of
  fiber type i.e., the general $B$-orbits are contained in the orbit
  closures of the $\TT$-action. We consider two cases.

  \smallskip\noindent\textbf{Case 1:} The general $\SL$-orbits on $X$
  are 2-dimensional. Then for general $x\in X$ one has
  $\overline{B\cdot x}=\overline{\SL\cdot x}$, and the $\SL$-action is
  of fiber type.

  \smallskip\noindent\textbf{Case 2:} The general $\SL$-orbits on $X$
  are 3-dimensional. Consider a general point $x\in X$ and the
  stabilizer $\TT^2_x\subseteq \TT$ of the subvariety
  $\overline{B\cdot x}$. Since any automorphism of the group $B$ is
  inner and the torus $\TT$ normalizes $B$, we may find a
  1-dimensional subtorus $S_x\subseteq \TT^2_x$ which commutes with
  the $B$-action on $\overline{B\cdot x}$. But the image of the
  homomorphism $\psi:\TT\rightarrow \Aut(\widetilde{\SL})\simeq \PSL$
  is a maximal torus, and so the subtorus $S_x$ is in the kernel of
  $\psi$. Hence, $S_x$ commutes with the $\SL$-action. In particular,
  $S_x$ preserves $B\cdot x$ and $\SL\cdot x$, and its action on
  $\SL\cdot x$ may be lifted to the action of a maximal torus
  $\widetilde{S}\subseteq \SL$ by right multiplication on $\SL$ via a
  finite covering $\widetilde{S}\rightarrow S_x$. But it is easy to
  check that the $(B\times \widetilde{S})$-action on $\SL$ has an open
  orbit, so $S_x$ permutes the general $B$-orbits on $\SL\cdot
  x$. This provides a contradiction.
\end{proof}

\subsection{$\SL$-actions on toric varieties}

In this section we give a complete classification of compatible
$\SL$-actions on affine toric varieties. Since a toric variety has an
open $\TT$-orbit, every $\SL$-action on a toric variety is of fiber
type.

\begin{definition} \label{def:sl-root}
  Let $\sigma\subseteq N_\QQ$ be a polyhedral cone. A root $e\in
  \RR(\sigma)$ is called an \emph{$\SL$-root} if also
  $-e\in\RR(\sigma)$.
\end{definition}

If $e$ is an $\SL$-root, then $\langle e,\rho_e\rangle=-1$, $\langle
e,\rho_{-e}\rangle=1$, and $\langle e,\rho\rangle=0\ \forall
\rho\in\sigma(1)\setminus \{\rho_{\pm e}\}$. Thus the number of
$\SL$-roots of a cone $\sigma$ with $r$ rays does not exceed $r(r-1)$,
and this bound is attained for a regular cone of dimension $r$.

\begin{theorem} \label{sl-toric} %
  The compatible $\SL$-actions on an affine toric variety $X_\sigma$
  are in bijection with the $\SL$-roots of $\sigma$. Furthermore, for
  every $\SL$-root $e\in \RR(\sigma)$, the corresponding $\SL$-action
  is effective if and only if the lattice vector $\rho_{-e}-\rho_{e}$
  is primitive.  If $\rho_{-e}-\rho_{e}$ is not primitive, then
  $\tfrac{1}{2}(\rho_{-e}-\rho_{e})$ is primitive and $\PSL$ acts
  effectively on $X_\sigma$.
\end{theorem}

\begin{proof}
  Let $A=\KK[\sigma^\vee_M]$ and $e\in \RR(\sigma)$ be an
  $\SL$-root. Letting $p=\rho_{-e}-\rho_e$, we define a $\ZZ$-grading
  on $A$ via
  $$\deg_\ZZ\chi^m=\langle m,p\rangle\in\ZZ,\quad\mbox{for all}
  \quad m\in \sigma_M^\vee\,.$$ %
  Hence, the infinitesimal generator of the corresponding $\GM$-action
  is given by
  $$\delta(\chi^m)=\langle m,p\rangle \chi^m,\quad\mbox{for all}
  \quad m\in \sigma_M^\vee\,.$$

  A routine computation shows that $\delta$, $\partial_e$ and
  $\partial_{-e}$ satisfy the conditions of
  Proposition~\ref{sl-actions}. Furthermore, since $\langle e,p\rangle=2$
  then $p$ is primitive or $\nicefrac{p}{2}$ is primitive. This proves
  the ``only if'' part of the theorem.

  To prove the converse, let $\delta,\partial_-,\partial_+$ be three
  homogeneous derivations satisfying the conditions of
  Proposition~\ref{sl-actions}. Since $\partial_\pm$ are LNDs, then
  $\partial_\pm=\lambda_\pm\partial_{e_\pm}$ for some $\lambda_\pm\in
  \KK^*$ and some roots $e_\pm\in \RR(\sigma)$. Furthermore, since the
  derivation $\delta$ comes from a downgrading of the $M$-grading on
  $A$, there is a lattice element $p$ such that
  $$\delta(\chi^m)=\langle m,p\rangle \chi^m,\quad\mbox{for all}
  \quad m\in \sigma_M^\vee\,.$$

  Since the commutator $[\partial_+,\partial_-]$ is a homogeneous
  operator of degree $e_++e_-$, we have $e:=e_+=-e_-$. One checks that
  the commutator is given by
  $$[\partial_+,\partial_-](\chi^m)=\lambda_+\lambda_-\langle m,
  \rho_{-e}-\rho_{e}\rangle \chi^m,\quad\mbox{for all} \quad m\in
  \sigma_M^\vee\,.$$ %
  Hence $p=\rho_{-e}-\rho_e$, $\lambda_+=\lambda_-^{-1}$, and the
  result follows.
\end{proof}

\begin{remark}
  If $e$ is an $\SL$-root of $\sigma$, then $-e$ is also an
  $\SL$-root. The corresponding $\SL$-actions are conjugate.
\end{remark}

\begin{example}\label{ex:toric-2}
  Let $X_\sigma$ be an affine toric variety of dimension 2. Up to
  automorphism of the lattice $N$ we may assume that $\sigma\subseteq
  N_\QQ$ is the cone spanned by the vectors $\rho_1=(1,0)$ and
  $\rho_2=(a,b)$, where $a\geq 0$, $a<b$, and $\gcd(a,b)=1$.  By
  Theorem~\ref{sl-toric} $X_\sigma$ admits a compatible $\SL$-action
  if and only if there exists $e\in M$ such that $\langle
  e,\rho_1\rangle=1$ and $\langle e,\rho_2\rangle=-1$. The only
  solution is $e=(1,-1)$ and $b=a+1$. Furthermore, the action is
  effective if and only if $b$ is odd.

  It is well known that the toric variety $X_\sigma$ corresponds to
  the affine cone over the rational normal curve $C$ of degree $a+1$
  (also known as Veronese cone). The curve $C$ is the image of $\PP^1$
  under the morphism
  $$\psi:\PP^1\hookrightarrow \PP^{a+1},\quad
  [x:y]\mapsto
  \left[x^{a+1}:x^{a}y:x^{a-1}y^2:\ldots:y^{a+1}\right]\,.$$ %
  The $\SL$-action on $X_\sigma$ is induced by the canonical
  $\SL$-action on the simple $\SL$-module $V(a+1)$ of binary forms of
  degree $a+1$.
\end{example}

\begin{example}
  Let now $X_\sigma$ be an affine toric variety of dimension
  3. Letting $e$ be an $\SL$-root of $\sigma$, we let $\rho_e$ and
  $\rho_{-e}$ be the corresponding distinguished rays and we consider
  a ray $\rho\neq \rho_{\pm e}$. Since $\langle e,\rho\rangle=0$ it
  follows that there are at most 2 non-distinguished rays. Thus the
  cone $\sigma$ has at most 4 rays.

  Assume first that $\sigma$ is simplicial and set $e=(1,0,0)$. Then,
  up to automorphism of the lattice $N$, the cone $\sigma$ is
  spanned by the vectors $\rho_1=(1,0,0)$, $\rho_2=(0,1,0)$, and
  $\rho_3=(-1,b,a)$, where $a>0$ and $0\leq b<a$.

  Let now $\sigma$ be a non simplicial cone and set again
  $e=(1,0,0)$. Then, up to automorphism of the lattice $N$, the
  cone $\sigma$ is spanned by the vectors $\rho_1=(1,0,0)$,
  $\rho_2=(0,1,0)$, $\rho_3=(0,b,a)$, and $\rho_4=(-1,c,d)$, where
  $a>0$, $0\leq b<a$, $\gcd(a,b)=1$, $d>0$, and $ac>bd$.
\end{example}

\begin{remark}
  In dimension 4 or greater, a cone admitting an $\SL$-root can have
  an arbitrary number of rays.
\end{remark}

\subsection{$\SL$-actions of fiber type on
  $\TT$-varieties} \label{sec:T-var-fib}

In the following theorem we give a classification of $\SL$-actions of
fiber type on normal affine $\TT$-varieties of arbitrary complexity.

\begin{theorem}\label{sec:sl-fiber}
  Let $\DD$ be a proper $\sigma$-polyhedral divisor on a
  semiprojective variety $Y$. Then the compatible $\SL$-actions of
  fiber type on the affine $\TT$-variety $X=X[Y,\DD]$ are in bijection
  with the $\SL$-roots $e$ of $\sigma$ such that the divisor $\DD(e)$ is
  principal and $\DD(e)+\DD(-e)=0$.

  Furthermore, the corresponding $\SL$-action is effective if and only
  if the lattice vector $\rho_{-e}-\rho_{e}$ is primitive. If
  $\rho_{-e}-\rho_{e}$ is not primitive, then
  $\tfrac{1}{2}(\rho_{-e}-\rho_{e})$ is primitive and $\PSL$ acts
  effectively on $X$.
\end{theorem}

\begin{proof}
  Letting $A=A[Y,\DD]=\KK[X]$, we let $e$ be an $\SL$-root of $\sigma$
  satisfying the conditions of the theorem. As in the toric case, we
  let $\rho_{\pm e}\in\sigma(1)$ be the distinguished rays of the
  roots $\pm e$, respectively. Letting $p=\rho_{-e}-\rho_e$, we define
  a $\ZZ$-grading on $A$ via
  $$\deg_\ZZ(A_m\cdot\chi^m)=\langle m,p\rangle\in\ZZ,\quad\mbox{for all}
  \quad m\in \sigma_M^\vee\,.$$ %
  So the infinitesimal generator of the corresponding $\GM$-action is
  given by
  $$\delta(f\chi^m)=\langle m,p\rangle\cdot f\chi^m,\quad\mbox{for all}
  \quad m\in \sigma_M^\vee\mbox{ and } f\in A_m\,.$$

  Letting $\varphi$ be any rational function on $Y$ such that
  $\divi(\varphi)+\DD(e)=0$, we let $\partial_+=\varphi\partial_e$ and
  $\partial_-=\varphi^{-1}\partial_{-e}$. The derivations
  $\partial_\pm$ are LNDs on $A$ by Theorem~\ref{lnd-fiber}.

  Now a routine computation shows that $\delta$, $\partial_+$ and
  $\partial_{-}$ satisfy the conditions of
  Proposition~\ref{sl-actions}. Furthermore, since $\langle e,p\rangle=2$
  then $p$ is primitive or $\nicefrac{p}{2}$ is primitive. This proves
  the ``only if'' part of the theorem.

  To prove the converse, let $\delta,\partial_-,\partial_+$ be three
  homogeneous derivations satisfying the conditions of
  Proposition~\ref{sl-actions}. Since $\partial_\pm$ are LNDs of fiber
  type, then $\partial_\pm=\partial_{e_\pm,\varphi_\pm}$ for some roots
  $e_\pm\in \RR(\sigma)$ and some $\varphi_\pm\in \Phi_{e_\pm}^*$.
  Similar to the toric case, we can prove that $e:=e_+=-e_-$.

  The commutator $[\partial_+,\partial_-]$ is given by
  $$[\partial_+,\partial_-](f\chi^m)=\varphi_+\varphi_-\langle m,
  p\rangle\cdot f\chi^m,\quad\mbox{for all} \quad m\in
  \sigma_M^\vee\mbox{ and } f\in A_m\,,$$ %
  where $p=\rho_{-e}-\rho_e$. Hence
  $\varphi_+=\varphi_-^{-1}$. Furthermore, since
  $\varphi_\pm\in\Phi_{e_\pm}^*$, we have
  \begin{equation}\label{eq-sl2-fib}
    \divi(\varphi_+)+\DD(e)\geq 0,\mbox{ and }\divi(\varphi_-)+\DD(-e)\geq
    0, \quad\mbox{so that}\quad \DD(e)+\DD(-e)\geq 0\,.
  \end{equation}

  Moreover,
  $$\DD(e)+\DD(-e)=
  \sum_Z\left(\min_i\{v_{i,Z}(e)\}-\max_i\{v_{i,Z}(e)\}\right)\cdot
  Z\leq 0\,.$$ %
  Hence $\DD(e)+\DD(-e)=0$. Finally, \eqref{eq-sl2-fib} yields
  $\divi(\varphi_+)+\DD(e)=0$ and so $\DD(e)$ is principal.
\end{proof}

\begin{remark}\label{rk:fiber}\
  \begin{enumerate}[$(1)$]
  \item By the proof of Theorem~\ref{sec:sl-fiber}, the condition
    $\DD(e)+\DD(-e)=0$ is fulfilled if and only if $v_{i,Z}(e)=v_{j,Z}(e)$
    for all prime divisors $Z\subseteq Y$ and all $i,j$.

  \item If $\rank M=2$ then the condition $\DD(e)+\DD(-e)=0$ in
    Theorem~\ref{sec:sl-fiber} can only be fulfilled if $\Delta_Z$ has
    only one vertex for all prime divisors $Z\subseteq Y$ i.e.,
    $\Delta_Z=v_Z+\sigma$. Indeed, the condition
    $v_{i,Z}(e)=v_{j,Z}(e)$ for all $i,j$ implies that all the
    vertices are contained in the line $L=\{v\in N_\QQ\mid \langle
    e,v-v_{1,Z}\rangle=0\}$. But $\pm e\notin \sigma^\vee$ and so
    $L\cap\sigma$ is a half line inside the cone $\sigma$. This
    implies that there can be only one vertex $v_Z:=v_{1,Z}$.
  \end{enumerate}
\end{remark}

\begin{example}
  Let $N=\ZZ^3$, $C=\AF^1$, and let $\sigma$ be the positive octant in
  $N_\QQ$. We also let $\Delta=\conv(v_1,v_2)+\sigma$, where
  $v_1=(1,1,-1)$ and $v_2=(-1,-1,1)$ and $\DD=\Delta\cdot[0]$. We
  consider the $\SL$-root $e=(-1,1,0)$ of $\sigma$. Since
  $v_1(e)=v_2(e)=0$ we have $\DD(e)+\DD(-e)=0$ and so by
  Theorem~\ref{sec:sl-fiber} the $\SL$-root $e$ produces an
  $\SL$-action on $X=X[C,\DD]$.

  The variety $X$ is toric by Corollary~\ref{AH-toric}. As a toric
  variety, $X$ is given by the non-simplicial cone $\tsigma\subseteq
  (N\oplus \ZZ)_\QQ$ spanned by $(v_1,1)$, $(v_2,1)$, $(\nu_1,0)$,
  $(\nu_2,0)$, and $(\nu_3,0)$, where $\{\nu_i\}$ is the standard base
  of $N$. The $\SL$-action is compatible with the big torus and is
  given by the $\SL$-root $\te=(e,0)$ of $\tsigma$.
\end{example}

\subsection{$\SL$-actions of horizontal type on
  $\TT$-varieties} \label{sec:T-var-hor}

In this section we give the more involved classification of
$\SL$-actions of horizontal type in the case of $\TT$-varieties of
complexity one. Here, we use the notation of Section~\ref{sec:LND-T}.

Letting $\DD$ be a proper $\sigma$-polyhedral divisor on the curve
$C=\AF^1$ or $C=\PP^1$, we let $X=X[C,\DD]$ and we assume that $X$
admits a compatible $\SL$-action of horizontal type. By
Proposition~\ref{sl-actions}, an $\SL$-action on $X$ is completely
determined by two homogeneous LNDs $\partial_\pm$ with
$\deg\partial_+=-\deg\partial_-=e$. By Theorem~\ref{lnd-hor}, the LNDs
of horizontal type are in bijection with coherent pairs. Let
$\partial_\pm$ be the LND given by the coherent pair $(\tDD_\pm,\pm
e)$, respectively, where
\begin{align} \label{eq:colorings}
  \tDD_\pm=\begin{cases}
    \{\DD;v_z^\pm,\forall z\in C\}\mbox{ with marked point }z_0^\pm & \mbox{if } C=\AF^1,  \\
    \{\DD,z_\infty^\pm;v_z^\pm,\forall z\in C\setminus
    z_\infty^\pm\}\mbox{ with marked point }z_0^\pm & \mbox{if } C=\PP^1\,.
  \end{cases}
\end{align}

Let $\AF^1=\spec\KK[t]$. In the sequel, we will assume $z_0^+=0$, and
$z_\infty^+=\infty$ if $C=\PP^1$. We also let $q(t)$ be a coordinate
around $z_0^-$ having point at infinity $z_\infty^-$ if $C=\PP^1$
i.e., $q$ is a M\"obious transformation
$$q(t)=\frac{at+b}{ct+d},\quad\mbox{with}\quad ad-bc=1,\quad q(z_0^-)=0
\quad\mbox{and}\quad q(z_\infty^-)=\infty \mbox{ if
}C=\PP^1\,.$$ %
In the case where $C=\AF^1$, we have $c=0$ and so we may choose
$a=d=1$ and $b=-z_0^-$.

By Corollary~\ref{cor:AH} we may and will assume $v_z^-=0$, for all
$z\in C\setminus\{z_0^-,z_\infty^-\}$. We also let $d^\pm$ be the smallest
positive integer such that $d^\pm\cdot v_{z_0^\pm}^\pm$ is contained
in the lattice $N$, and $s^\pm=-\tfrac{1}{d^\pm}\mp
v_{z_0^\pm}^\pm(e)$. In this setting \eqref{eq:1} yields
\begin{align}
  \label{eq:LND-}
  \partial_-(\chi^m\cdot q^r)=d^-(v_{z_0^-}^-(m)+r)\cdot\chi^{m-e}\cdot
  q^{r+s^-}, \quad\mbox{for all}\quad
  (m,r)\in M\oplus \ZZ\,.
\end{align}

To obtain a similar expression for $\partial_+$, we let
$$\DD'=\begin{cases}
  \DD-\sum_{z\neq z_0^+}(v_z^++\sigma)\cdot z & \mbox{if }
  C=\AF^1, \\
  \DD-\sum_{z\neq z_0^+,z_\infty^+}(v_z^++\sigma)\cdot (z-z_\infty^+) & \mbox{if }
  C=\PP^1\,.
\end{cases}
$$

By Corollary~\ref{cor:AH}, $X[C,\DD]\simeq X[C,\DD']$ equivariantly and
in the new $\sigma$-polyhedral divisor $\DD'$ the colored vertices
$v_z^+$ are zero for all $z\neq 0$. Furthermore, letting
\begin{align*}
&A[C,\DD]=\bigoplus_{m\in\sigma^\vee_M}A_m\chi^m,
\quad\mbox{where}\quad A_m=H^0(C,\OO(\DD(m))),\quad\mbox{and} \\
&A[C,\DD']=\bigoplus_{m\in\sigma^\vee_M}A'_m\xi^m,
\quad\mbox{where}\quad A'_m=H^0(C,\OO(\DD'(m)))\,, 
\end{align*}
by Theorem~\ref{AH} the isomorphism $A[C,\DD]\rightarrow A[C,\DD']$ is
given by $\xi^m=\varphi^m\chi^m$, where $\varphi^m\in \KK(t)$ is a
rational function whose divisor is $\DD'(m)-\DD(m)$ for all
$m\in\sigma^\vee_M$, and
$\varphi^m\cdot\varphi^{m'}=\varphi^{m+m'}$. In this setting
\eqref{eq:1} yields
\begin{align}
  \label{eq:LND+}
  \partial_+(\varphi^m\chi^m\cdot
  t^r)=d^+(v_{0}^+(m)+r)\cdot\varphi^{m+e}\chi^{m+e}\cdot t^{r+s^+},
  \quad\mbox{for all}\quad (m,r)\in M\oplus \ZZ\,.
\end{align}

Recall that the LNDs $\partial_\pm$ on $\KK[X]$ correspond to the
$U_\pm$-actions on $X$ of a compatible $\SL$-action on $X$ and so by
Corollary~\ref{cor-downg} $(i)$, the commutator
$\delta=[\partial_+,\partial_-]$ is a downgrading of the $M$-grading
on $\KK[X]$ i.e., there exits $p\in N$ such that
\begin{align}\label{eq:comm}
  \delta(f\chi^m)=\langle m,p\rangle\cdot f\chi^m, \quad\mbox{for
    all}\quad m\in\sigma^\vee_M\mbox{ and } f\in \KK(t)\,.
\end{align}

To lighten the notation, we use the ``prime notation'' to denote the
partial derivative with respect to $t$ i.e., $\tfrac{d}{dt}(f)=f'$.

\begin{proposition} \label{pr:hor-simple} %
  If $X[C,\DD]$ admits an $\SL$-action of horizontal type, then the
  marked points and the infinity points (if $C=\PP^1$) of $\tDD_+$ and
  $\tDD_-$ can be chosen to be equal i.e., in the notation above,
  without loss of generality, we may assume $z_0^+=z_0^-=0$ and
  $z_\infty^+=z_\infty^-=\infty$. Moreover, we have $d^+=d^-:=d$.
\end{proposition}

\begin{proof}
  By \eqref{eq:comm} we have $\delta(t)=0$ and a routine computation
  (see the appendix) shows that
  \begin{align}\label{commutator-1}
    \delta(t)=d^+d^-\cdot\varphi^e\cdot t^{s^+}\cdot
    q^{s^-}\cdot\left(\left(1-\tfrac{1}{d^-}\right)t-
      \left(1-\tfrac{1}{d^+}\right)\frac{q}{q'}-\frac{q''q
        t}{(q')^2}\right)\,,
  \end{align}
  and so
  $$\Gamma:=\left(\left(1-\tfrac{1}{d^-}\right)t-
    \left(1-\tfrac{1}{d^+}\right)\frac{q}{q'}-\frac{q''q
      t}{(q')^2}\right)=0\,.$$ %

  Recall that $q(t)=\tfrac{at+b}{ct+d}$ with $ad-bc=1$. Letting
  $\ell^\pm=1-\nicefrac{1}{d^\pm}$, a simple computation shows that
  $$\Gamma=ac(2-\ell^+)t^2+(\ell^--\ell^+(2bc+1)+2bc)t+\ell^+bd=0\,.$$
  Since $\ell^\pm<1$, we have $ac=0$. If $a=0$, then $bc=-1$ and so
  $\ell^++\ell^-=2$. This provides a contradiction. Hence, $c=0$ so
  $ad=1$ and $\ell^+=\ell^-$. This last equality gives
  $d^+=d^-$. Furthermore, the equality $c=0$ yields
  $z_\infty^-=z_\infty^+=\infty$. Hence, we may assume
  $q(t)=t-z_0^-$ and the commutator becomes
  $$\delta(t)=d^+d^-\cdot\varphi^e\cdot t^{s^+}\cdot
  (t-z_0^-)^{s^-}\cdot\Big(\left(1-\tfrac{1}{d^-}\right)t-
  \left(1-\tfrac{1}{d^+}\right)(t-z_0^-)\Big)\,.$$ %

  Assume for a moment that $z_0^-\neq z_0^+=0$. Then $\delta(t)=0$
  implies $d^+=d_-=1$ i.e., the colored vertices $v_{z_0^+}^+$ and
  $v_{z_0^-}^-$ of the respective marked point belong to the lattice
  $N$. In this case Definition~\ref{def:col-pol} shows that there are
  no marked points and we can choose $z_0^+=z_0^-$ to be any point
  different from the common point at infinity.
\end{proof}

By the previous proposition, in the sequel we assume $z_0^+=z_0^-=0$,
$z_\infty^+=z_\infty^-=\infty$, and $d^+=d^-:=d$ so that the LNDs
$\partial_+$ and $\partial_-$ are given by
\begin{align*}
  \partial_+(\varphi^m\chi^m\cdot
  t^r)=d\cdot(v_{0}^+(m)+r)\cdot\varphi^{m+e}\chi^{m+e}\cdot t^{r+s^+},
  \quad\mbox{for all}\quad (m,r)\in M\oplus \ZZ\,. \\
  \partial_-(\chi^m\cdot t^r)=d\cdot(v_{0}^-(m)+r)\cdot\chi^{m-e}\cdot
  t^{r+s^-}, \quad\mbox{for all}\quad
  (m,r)\in M\oplus \ZZ\,. \\
  s^+=-\nicefrac{1}{d}-v_{0}^+(e),\quad
  s^-=-\nicefrac{1}{d}+v_{0}^-(e)\quad\mbox{and}\quad
  \varphi^m=\prod_{z\neq 0,\infty} (t-z)^{-v_z^+(m)},\ \forall m \in
  \sigma_M^\vee\,.
\end{align*}

\begin{corollary} \label{cor-tor-hor} %
  Let $X=X[C,\DD]$ be a $\TT$-variety of complexity one endowed with a
  compatible $\SL$-action of horizontal type. If $\Delta_z=\sigma$ for
  all $z\neq z_0^\pm,z_\infty^\pm$, then $X$ is toric and the $\SL$-action is
  compatible with the big torus.
\end{corollary}
\begin{proof}
  The variety $X$ is toric by Corollary~\ref{AH-toric}. Furthermore,
  the big torus action is induced by the $(M\oplus \ZZ)$-grading of
  $\KK[X]$ given by $\deg(\chi^m)=(m,0)$ and $\deg(t)=(0,1)$. Since
  $\Delta_z=0$ for all $z\in \AF^1\setminus\{0\}$, we have
  $\varphi^m=1$ for all $m\in \sigma_M^\vee$. Hence, by \eqref{eq:2}
  $\partial_\pm$ are homogeneous with respect to the $(M\oplus
  \ZZ)$-grading of $\KK[X]$. This gives that the $U_\pm$-actions on
  $X$ are compatible with the big torus action, and so does the
  $\SL$-action.
\end{proof}

Since compatible $\SL$-actions on toric varieties are described in
Theorem~\ref{sl-toric}, in the sequel we restrict to the case where
the $\SL$-action on $X$ is not compatible with a bigger torus. In the
next lemma we show that if $X[C,\DD]$ admits an $\SL$-action of
horizontal type, then $\DD$ has a very special form. 

For a subset $S\subseteq N_{\QQ}$ we denote the convex hull of $S$ by
$\conv(S)$. For a vector $e\in M_{\QQ}$ we let $e^\bot=\{p\in
N_{\QQ}\mid \langle e,p\rangle=0\}$ be the subspace of $N_{\QQ}$
orthogonal to $e$.

\begin{lemma} \label{lm:hor-DD} Let $X=X[C,\DD]$ be a normal affine
  $\TT$-variety of complexity one endowed with a compatible
  $\SL$-action of horizontal type. Assume that the $\SL$-action is not
  compatible with a bigger torus and let $e\in M$ be the degree of the
  homogeneous LND $\partial_+$ on $\KK[X]$ corresponding to the
  $U_+$-action on $X$. Then $C=\AF^1$ or $C=\PP^1$, $\DD=\sum_{z\in
    C}\Delta_z\cdot z$, and the $\sigma$-polyhedra $\Delta_z$ can be
  chosen (via Corollary~\ref{cor:AH}) as in one of the following
  cases:
  $$\Delta_0=\conv(0,v_0^-)+\sigma,\quad
  \Delta_1=\conv(0,v_1^+)+\sigma, \quad \Delta_z=\sigma\ \forall z\in
  \AF^1\setminus\{0,1\},\quad\mbox{and}\quad
  \Delta_\infty=\Pi+\sigma\,,$$ %
  where $v_0^-,v_1^+\in N$, $v_0^-,v_1^+\notin \sigma$, $v_0^-(e)=1$,
  $v_1^+(e)=-1$, and $\Pi\subseteq e^\bot$ is
  a bounded polyhedron; or
  $$\Delta_0=v_0^-+\sigma,\quad
  \Delta_1=\conv(0,v_1^+)+\sigma, \quad \Delta_z=\sigma\ \forall z\in
  \AF^1\setminus\{0,1\},\quad\mbox{and}\quad
  \Delta_\infty=\Pi+\sigma\,,$$ %
  where $2v_0^-,v_1^+\in N$, $v_1^+\notin \sigma$, $2v_0^-(e)=1$, $v_1^+(e)=-1$, and
  $\Pi\subseteq e^\bot$
  is a bounded polyhedron.
\end{lemma}

For the proof of this lemma we need the following notation:
$$\alpha_m=t\frac{d}{dt}\left(\ln(\varphi^m)\right), \quad v_0=v_0^--v_0^+
\quad\mbox{and}\quad \nu=v_0(e)-\nicefrac{1}{d}\,.$$ %
With this definition $\alpha_{m+m'}=\alpha_m+\alpha_{m'}$. %
More explicitly,
\begin{align*}
\alpha_m=-t\sum_{z\neq 0,\infty}v_z^+(m)\frac{1}{t-z}=
-v^+(m)-\sum_{z\neq 0,\infty}v_z^+(m)\frac{z}{t-z},
\quad\mbox{where}\\ v^+=\sum_{z\neq 0,\infty}v_z^+=v_{\deg}^+-v_0^+, 
\quad\mbox{and}\quad \alpha'_m=\tfrac{d}{dt}(\alpha_m)=\sum_{z\neq
  0,\infty}v_z^+(m)\frac{z}{(t-z)^2}\,.
\end{align*}

For a rational function $R(t)=\nicefrac{P(t)}{Q(t)}$, we define the
degree $\deg R=\deg P-\deg Q$ so that $\deg(R_1\cdot
R_2)=\deg(R_1)+\deg(R_2)$. We also let the \emph{principal part} of
$R$ be the result of the polynomial division between $P$ and $Q$. Then
$\deg(R)=0$ if and only if the principal part of $R$ is a non-zero
constant.

\begin{proof}[Proof of Lemma~\ref{lm:hor-DD}]
  The appendix shows that the commutator
  $\delta=[\partial_+,\partial_-]$ is given by
  \begin{align}
    \label{commutator-2}
    \delta(\chi^m t^r)=d^2\varphi^e t^{\nu-\nicefrac{1}{d}}
    \cdot\left(\nu
      v_0(m)+\alpha_ev_0(m)+\nu\alpha_m+t\alpha'_m+\alpha_e\alpha_m
    \right)\cdot \chi^m t^r:=\Gamma\cdot \chi^m t^r\,.
  \end{align}
  Thus, by \eqref{eq:comm} the expression $\Gamma$ has to be
  independent of $t$, linear in $m$, and $\Gamma\not\equiv 0$.

  Assume that $v_{z}^+\neq 0$ and $v_{z}^+(e)=0$ for some $z\neq
  0,\infty$. Then $\varphi^e$ does not contain the factor $(t-z)$ and
  for any $m$ such that $v_{z}^+(m)\neq 0$ the summand
  $v_z^+(m)\tfrac{zt}{(t-z)^2}$ in $t\alpha'_m$ cannot be eliminated
  since $\alpha_m$ and $t\alpha'_m$ are linearly independent. Hence
  $v_{z}^+(e)=0$ implies $v_{z}^+=0$. Moreover, we have
  $v_z^+(e)\in\{0,-1,-2\}$ since otherwise, the factor
  $(t-z)^{-v_z^+(e)}$ in $\varphi^e$ cannot be canceled in
  $\Gamma$. Hence, $\varphi^e$ is a polynomial.
 
  A direct computation shows that the principal part of $\nu
  v_0(m)+\alpha_ev_0(m)+\nu\alpha_m+t\alpha'_m+\alpha_e\alpha_m$ is
  given by
  $$L:=(v_0(e)-v^+(e)-\nicefrac{1}{d})\cdot(v_0(m)-v^+(m))\,. $$

  Assume that $L(e)=0$. Since
  $\deg(\varphi^et^{\nu-\nicefrac{1}{d}})=\nu-\nicefrac{1}{d}-v^+(e)
  =v_0(e)-v^+(e)-\nicefrac{2}{d}$, we have
  $\deg(\varphi^et^{\nu-\nicefrac{1}{d}})<0$ and so
  $\deg(\Gamma)<0$. This is a contradiction since $\Gamma(e)$ has to
  be independent of $t$. In the following, we assume $L(e)\neq
  0$. This yields $\deg(\varphi^et^{\nu-\nicefrac{1}{d}})=0$.

  We divide the proof in the following three cases:
  $$\mbox{Case I\,: }\nu v_0\neq 0, \quad \mbox{Case II\,: }\nu=0,
  \quad\mbox{and}\quad \mbox{Case III\,: }v_0=0\,.$$

  \medskip\noindent{\bf Case I\,:} Evaluating $\Gamma$ in $t=0$ we
  obtain $\Gamma=d^2\varphi^e(0)\cdot 0^{\nu-\nicefrac{1}{d}}\cdot \nu
  v_0(m)$. Hence, we have $\nu-\nicefrac{1}{d}=0$ and since
  $\deg(\varphi^et^{\nu-\nicefrac{1}{d}})=0$, we have
  $\varphi^e=1$. This yields $v_{z}^+(e)=0$ for all $z\neq 0,\infty$
  and so $v_{z}^+=0$ for all $z\neq 0,\infty$.

  Let $z\neq 0,\infty$ and assume that $\Delta_z$ has a vertex $v\neq
  0$. Since $v_{z}^\pm=0$ by Definition~\ref{col-div} (2) applied to
  $\tDD_\pm$ we obtain $v(e)\geq 1$ and $-v(e)\geq 1$ which provides a
  contradiction. This yields $\Delta_z=\sigma$. Thus, $X[C,\DD]$ is a
  toric variety and by Corollary~\ref{AH-toric} the $\SL$-action is
  compatible with the big torus.

  \medskip\noindent{\bf Case II\,:} The condition $\nu=0$ implies
  $d=1$ since $\nu-\nicefrac{1}{d}$ appears as the exponent of $t$ in
  $\Gamma$. This yields $v_0^\pm\in N$ and so we can assume $v_0^+=0$
  by Corollary~\ref{cor:AH}. Now $\nu=v_0^-(e)-1$ and so
  $v_0^-(e)=1$. Furthermore,
  $\deg(\varphi^et^{\nu-\nicefrac{1}{d}})=0$ implies
  $\deg(\varphi^e)=1$ and so we can assume $v^+_z(e)=0$ for all $z\neq
  1$ and $v^+_1(e)=-1$. This yields $v^+_z=0$ for all $z\neq 1$. Now, the
  commutator is given by 
  \begin{align*}
    \delta(\chi^m t^r)=\frac{t-1}{t}
    \cdot\left(v_0^-(m)\frac{t}{t-1}+v_1^+(m)\frac{t}{(t-1)^2}-
      v_1^+(m)\frac{t^2}{(t-1)^2} \right)\cdot \chi^m t^r\,.
  \end{align*}

  Since
  \begin{align} \label{eq:rational}
    \frac{t}{t-1}+\frac{t}{(t-1)^2}-\frac{t^2}{(t-1)^2}=0\,,
  \end{align}
  we have
  $$\delta(\chi^m t^r)=\langle m,v_0^--v_1^+\rangle \cdot \chi^m t^r,
  \quad\mbox{for all}\quad (m,r)\in M\oplus \ZZ\,.$$ %
  
  Let now $z\neq 0,1,\infty$ and assume that $\Delta_z$ has a vertex
  $v\neq 0$. Since $v_{z}^\pm=0$ by Definition~\ref{col-div} (2)
  applied to $\tDD_\pm$ we obtain $v(e)\geq 1$ and $-v(e)\geq 1$ which
  provides a contradiction. Thus $\Delta_z=\sigma$. A similar argument
  shows that the only vertices in $\Delta_0$ and $\Delta_1$ are
  $\{0,v_0^-\}$ and $\{0,v_1^+\}$, respectively. Finally, if
  $C=\PP^1$, let $v$ be a vertex of
  $\Delta_\infty$. Definition~\ref{col-div} (4) shows that $v(e)\geq
  0$ and $v(e)\geq 0$, so that $-v(e)=0$. This corresponds to the
  first case in the lemma.

  \medskip\noindent{\bf Case III\,:} The condition $v_0=0$ implies
  $v_0^-=v_0^+$ and $\nu=-\nicefrac{1}{d}$. Hence $d=1$ or $d=2$ since
  $\nu-\nicefrac{1}{d}=-\nicefrac{2}{d}$ appears as the exponent of
  $t$ in $\Gamma$. If $d=1$, then by Definition~\ref{def:col-pol}, we
  can change the marked points of $\tDD_\pm$ so that $v_{z_0}^-\neq
  v_{z_0}^+$. Hence, this case reduces to Case I or Case II.
  
  Assume now that $d=2$ so that $v_0^-=v_0^+\in \tfrac{1}{2}N\setminus
  N$. The condition $\deg(\varphi^et^{\nu-\nicefrac{1}{d}})=0$ implies
  $\deg(\varphi^e)=1$ and so we can assume $v^+_z(e)=0$ for all $z\neq
  1$ and $v^+_1(e)=-1$. This yields $v^+_z=0$ for all $z\neq 1$. The
  commutator is now given by
  \begin{align*}
    \delta(\chi^m t^r)=2v_1^+(m)\frac{t-1}{t}
    \cdot\left(\frac{t}{t-1}+2\frac{t}{(t-1)^2}-
      2\frac{t^2}{(t-1)^2} \right)\cdot \chi^m t^r\,.
  \end{align*}
  By \eqref{eq:rational} we have
  $$\delta(\chi^m t^r)=\langle m,-2v_1^+\rangle \cdot \chi^m t^r,
  \quad\mbox{for all}\quad (m,r)\in M\oplus \ZZ\,.$$ %

  By the same argument as in Case II we obtain $\Delta_z=\sigma$ for
  all $z\neq 0,1,\infty$ and the only vertices in $\Delta_0$ and
  $\Delta_1$ are $\{v_0^-\}$ and $\{0,v_1^+\}$, respectively. Finally,
  if $C=\PP^1$, let $v$ be a vertex of $\Delta_\infty$. By
  Corollary~\ref{cor:AH}, we can assume that $2v_0^-(e)=1$, and
  Definition~\ref{col-div} (4) shows that for every vertex $v$ of
  $\Delta_\infty$, we have $v(e)=0$. This corresponds to the second
  case in the lemma. The proof is now completed.
\end{proof}

To obtain a full classification of compatible $\SL$-actions of
horizontal type on $X$ we only need conditions for existence of the
homogeneous LNDs $\partial_\pm$ of horizontal type on $\KK[X]$ defined
in \eqref{eq:LND-} and \eqref{eq:LND+}.

The following theorem provides the announced classification of
compatible $\SL$-actions of horizontal type on $\TT$-varieties of
complexity one. For the theorem, we need the following notation. Let
$\DD$ be as in Lemma~\ref{lm:hor-DD}. Then $\DD$ admits two different
colorings as in \eqref{eq:colorings}.  We let $\tomega_\pm$ be the
associated cone of $\tDD_\pm$ (see before
Definition~\ref{col-div}). Furthermore, for every $e\in M$, we let
$\te_\pm=(\pm e,-\nicefrac{1}{d}\mp v_0^\pm(e))\in M\oplus \ZZ$.

\begin{theorem} \label{th:hor} %
  Let $X=X[C,\DD]$ be a normal affine $\TT$-variety of complexity
  one. Then $X$ admits a compatible $\SL$-action of horizontal type
  that is not compatible with a bigger torus if and only if the
  following conditions hold.
  \begin{enumerate}[(i)]
  \item The base curve $C$ is either $\AF^1$ or $\PP^1$.
  \item There exists a lattice vector $e\in M$ such that the
    $\sigma$-polyhedral divisor $\DD$ may be shifted via
    Corollary~\ref{cor:AH} to one of the following two forms
    $$\Delta_0=\conv(0,v_0^-)+\sigma,\quad
    \Delta_1=\conv(0,v_1^+)+\sigma, \quad \Delta_z=\sigma\ \forall
    z\in \AF^1\setminus\{0,1\},\quad\mbox{and}\quad
    \Delta_\infty=\Pi+\sigma\,,$$ %
    where $v_0^-,v_1^+\in N$, $v_0^-,v_1^+\notin \sigma$,
    $v_0^-(e)=1$, $v_1^+(e)=-1$, and $\Pi\subseteq e^\bot$ is a
    bounded polyhedron; or
    $$\Delta_0=v_0^-+\sigma,\quad
    \Delta_1=\conv(0,v_1^+)+\sigma, \quad \Delta_z=\sigma\ \forall
    z\in \AF^1\setminus\{0,1\},\quad\mbox{and}\quad
    \Delta_\infty=\Pi+\sigma\,,$$ %
    where $2v_0^-,v_1^+\in N$, $v_1^+\notin \sigma$, $2v_0^-(e)=1$,
    $v_1^+(e)=-1$, and $\Pi\subseteq e^\bot$ is a bounded polyhedron.
  \item The lattice vectors $\te_\pm\in M\oplus\ZZ$ are roots of the
    cones $\tomega_\pm$, respectively.
  \end{enumerate}

  Moreover, if $(C,\sigma,\DD)$ is in one of the two forms above, then
  the compatible $\SL$-action of horizontal type on $X$ is given by
  the $\slt$-triple $\{\delta,\partial_+,\partial_-\}$ of derivations,
  where $\delta=[\partial_+,\partial_-]$, the homogeneous LNDs
  $\partial_\pm$ are given by the coherent pairs $(\tDD_\pm,\pm e)$,
  and $\tDD_\pm$ are the following colorings of $\DD$
  \begin{align*}
    &\begin{cases}
      \tDD_+=\{\DD,\infty;v_1=v_1^+, v_z=0,\forall z\neq 1,\infty\},\\
      \tDD_-=\{\DD,\infty;v_0=v_0^-, v_z=0,\forall z\neq
      0,\infty\}\,,
    \end{cases} \mbox{in the first case; or}\\
    &\begin{cases}
      \tDD_+=\{\DD,\infty;v_0=v_0^-,v_1=v_1^+, v_z=0,\forall z\neq
      0,1,\infty\}, \\
      \tDD_-=\{\DD,\infty;v_0=v_0^-, v_z=0,\forall
      z\neq 0,\infty\},
    \end{cases} \mbox{in the second case.}
  \end{align*}
\end{theorem}

\begin{proof}
  By Lemma~\ref{lm:hor-DD}, if $X$ admits a compatible $\SL$-action of
  horizontal type that is not compatible with a bigger torus, then (i)
  and (ii) hold. Moreover, by the proof of Lemma~\ref{lm:hor-DD}, such
  an $X$ admits a compatible $\SL$-action of horizontal type if and
  only if the derivations $\partial_\pm$ given by \eqref{eq:LND-} and
  \eqref{eq:LND+} define LNDs on $\KK[X]$.

  By Theorem~\ref{lnd-hor}, the derivations $\partial_\pm$ define LNDs
  on $\KK[X]$ if and only if there exists $e\in M$ such that
  $(\tDD_\pm,\pm e)$ are coherent pairs. Furthermore, $(\tDD_\pm,\pm
  e)$ are coherent pairs if and only if $\te_\pm$ is a root of the
  cone $\tomega_\pm$ and Definition~\ref{col-div} (2)--(4) hold. It is
  a routine verification that Definition~\ref{col-div} (2)--(4) hold
  for $(\tDD_\pm,\pm e)$, and so the theorem is proved.
\end{proof}

\section{Special $\SL$-actions}
\label{sec:special}

In this section we give a classification of special $\SL$-actions on
normal affine varieties. This generalizes Theorem 1 in
\cite{Arz97}. Let us first state the necessary definitions and results
for an arbitrary reductive group.

Let $G$ be a connected reductive algebraic group, $T\subseteq B$ be a
maximal torus and a Borel subgroup of $G$, and $\fX_+(G)$ be the
semigroup of dominant weights of $G$ with respect to the pair
$(T,B)$. Any regular action of the group $G$ on an affine variety $X$
defines a structure of rational $G$-algebra on the algebra of regular
functions $\KK[X]$. In particular, we have the isotypic decomposition
$$\KK[X]=\bigoplus_{\lambda\in\fX_+(G)} \KK[X]_\lambda\,,$$
where $\KK[X]_\lambda$ is the sum of all the simple $G$-submodules in
$\KK[X]$ with the highest weight $\lambda$.

\begin{definition}
  A $G$-action on $X$ is called \emph{special} (or
  \emph{horospherical}), if there exists a dense open $W\subseteq X$
  such that the isotropy group of any point $x\in W$ contains a
  maximal unipotent subgroup of the group $G$.
\end{definition}

\begin{remark}
  If a $G$-action is special, then the isotropy group $G_x$ contains a
  maximal unipotent subgroup for all $x\in X$.
\end{remark}

\begin{theorem}[See {\cite[Theorem~5]{Pop86}}]
  A $G$-action on an affine variety $X$ is special if and only if
  $$\KK[X]_\lambda\cdot \KK[X]_\mu\subseteq \KK[X]_{\lambda+\mu},
\quad \mbox{for all} \quad \lambda,\mu\in\fX_+(G)\,.$$
\end{theorem}

\begin{corollary} \label{cor:torus-S} %
  For a special action, the isotypic decomposition is a
  $\fX_+(G)$-grading on the algebra $\KK[X]$. This defines an action
  of an algebraic torus $S$ on $X$, and this action commutes with the
  $G$-action.
\end{corollary}

Furthermore, since $S$ acts on every isotypic component by scalar
multiplication, every $G$-invariant subspace in $\KK[X]$ is
$S$-invariant. In particular, $S$ preserves every $G$-invariant ideal
in $\KK[X]$, and thus every $G$-invariant closed subvariety in
$X$. This shows that the torus $S$ preserves all $G$-orbit closures on
$X$.

We return now to the case of $\SL$-actions on $\TT$-varieties.

\begin{proposition}\label{prp:fiber-special}
  Every compatible $\SL$-action of fiber type on an affine
  $\TT$-variety $X$ is special.
\end{proposition}
\begin{proof}
  For a general $x\in X$, let $Y=\overline{\SL\cdot x}$. Then
  $Y\subseteq \overline{\TT\cdot x}$. Denote by $\TT_Y$ the stabilizer
  of the subvariety $Y$ in $\TT$. Since the torus $\TT$ normalizes the
  $\SL$-action, it permutes $\SL$-orbit closures, and thus $\TT_Y$
  acts on $Y$ with an open orbit.

  Since $\TT_Y$ also normalizes the $\SL$-action, there exists a
  subtorus $S_Y\subseteq \TT_Y$ of codimension 1 that centralizes the
  $\SL$-action on $Y$. In particular, it preserves the open orbit
  $\SL\cdot x\hookrightarrow Y$. But $\SL\cdot x\simeq \SL/H$, where
  $H$ is the isotropy group of $x$ in $\SL$. We have $S_Y\subseteq
  \Aut_{\SL}(\SL/H)$ and thus
  $$\rank \Aut_{\SL}(\SL/H)\geq \dim(\SL/H)-1\,.$$
  But $\Aut_{\SL}(\SL/H)\simeq N_{\SL}(H)/H$ and $\rank
  N_{\SL}(H)/H\leq 1$, so $\dim(\SL/H)\leq 2$. If $H$ coincides either
  with a maximal torus or with its normalizer in $\SL$, then the group
  $N_{\SL}(H)/H$ is finite, a contradiction. So, $H$ is a finite
  extension of a maximal unipotent subgroup of $\SL$, and the
  $\SL$-action is special. 
\end{proof}

\begin{corollary}\label{cor:tor-special}
Every compatible $\SL$-action on a toric variety is special.  
\end{corollary}

In the following proposition we come to a partial converse of
Proposition~\ref{prp:fiber-special}. Namely, we realize any special
action of $\SL$ as a compatible action of fiber type with respect to a
canonical 2-dimensional torus action.

\begin{proposition}\label{lm:SU}
  Every special $\SL$-variety admits the action of a 2-dimensional
  torus such that the $\SL$-action is compatible with the torus action
  and of fiber type.
\end{proposition}

\begin{proof}
  Let $\TT^2$ be the 2-dimensional torus $T\cdot S$, where $T$ is a
  maximal torus in $\SL$ and $S$ is one constructed in
  Corollary~\ref{cor:torus-S}. By construction, the actions of $T$ and
  $S$ on $X$ commute, preserve every $\SL$-orbit closure and $\TT^2$
  has an open orbit on every such orbit closure.
\end{proof}

In the following proposition we determine the special $\SL$-actions
among the compatible $\SL$-actions on a complexity one affine
$\TT$-variety.

\begin{proposition} \label{prop:non-special} %
  Let $X$ be a normal affine $\TT$-variety of complexity one endowed
  with a compatible $\SL$-action. Then the $\SL$-action is special if
  and only if it is either of fiber type; or it is of horizontal type, $X$ is
  toric, and the $\SL$-action is compatible with the big torus. In
  particular, the $\TT$-varieties of complexity one that admit a
  non-special compatible $\SL$-action are given in
  Theorem~\ref{th:hor}.
\end{proposition}

\begin{proof}
  If the $\SL$-action is of fiber type, then the proposition follows
  from Proposition \ref{prp:fiber-special}.

  Assume that the $\SL$-action is of horizontal type and
  special. Since the $\SL$-action is compatible, $\TT$ is a product of
  a maximal torus $T$ of $\SL$ and a subtorus $T'$ which commutes
  with the $\SL$-action. In particular, $\TT$ preserves all the
  $\SL$-isotypic components in $\KK[X]$.

  On the other hand, the one-dimensional torus $S$ constructed in
  Corollary~\ref{cor:torus-S} acts on any isotypic component by a
  scalar multiplication. Thus $S$ commutes with $\TT$. We know that
  general closures of the canonical 2-torus $(T\cdot S)$-orbits
  coincide with closures of $\SL$-orbits (see
  Proposition~\ref{lm:SU}). Since the $\SL$-action is compatible, $T$
  is contained in $\TT$ and since the $\SL$-action is of horizontal,
  the closures of $\SL$-orbits are not contained in the closures of
  the $\TT$-orbits. Hence $S$ is not contained in $\TT$, so we may
  extent $\TT$ by $S$ and get a big torus which acts on X with an open
  orbit.
\end{proof}

\smallskip

For the rest of this section we let $\TT^2$ be a 2-dimensional
algebraic torus. In the following, we give a description of compatible
$\SL$-actions of fiber type on $\TT^2$-varieties. By Propositions
\ref{prp:fiber-special} and \ref{lm:SU}, this gives a description of
all special $\SL$-actions on normal affine varieties.

The following example gives a construction of certain
$\TT^2$-varieties admitting a compatible $\SL$-action of fiber type.

\begin{example}\label{ex:T-var-U}
  Let $M$ be a lattice of rank 2, and $\sigma$ be the cone spanned in
  $N_\QQ$ by the vectors $(1,0)$ and $(r-1,r)$, for some
  $r\in\ZZ_{>0}$. By Example~\ref{ex:toric-2} the cone $\sigma$
  admits the $\SL$-root $e=(1,-1)$.

  We also fix a semiprojective variety $Y$ and an ample $\QQ$-Cartier
  divisor $H$ on $Y$. Consider the $\sigma$-polyhedral divisor given
  by $\DD=\Delta\cdot H$, where $\Delta=(1,1)+\sigma$. The
  $\sigma$-polyhedral divisor $\DD$ is proper since for every
  $(m_1,m_2)\in\sigma^\vee_M\setminus \{0\}$ the evaluation divisor is
  given by $\DD(m_1,m_2)=(m_1+m_2)\cdot H$ and $m_1+m_2>0$.

  We have $\DD(e)=\DD(-e)=0$ and so Theorem~\ref{sec:sl-fiber} yields that
  the $\TT^2$-variety $X=X[Y,\DD]$ admits an $\SL$-action of fiber
  type and the generic isotropy subgroup in $\SL$ is $U_{(r)}$.

  Furthermore, if we let $X'$ be the $\TT^2$-variety obtained with the
  above construction with the data $Y$ and $H$ replaced by $Y'$ and
  $H'$, then $X$ is isomorphic to $X'$ if and only if $Y\simeq Y'$ and
  under this isomorphism $H$ is linearly equivalent to $H'$. Indeed,
  since $H$ is ample \cite[Proposition~3.3]{Dem88} implies that $Y$ is
  unique up to isomorphism. Finally, Corollary~\ref{cor:AH} shows that
  $H$ and $H'$ are linearly equivalent.
\end{example}

\begin{proposition} \label{lm:T2-ex}
  Every $\TT^2$-variety $X$ endowed with an $\SL$-action of fiber type
  is isomorphic to one in Example~\ref{ex:T-var-U} above.
\end{proposition}
\begin{proof}
  Let $X=X[Y,\DD]$ where $\DD=\sum_Z \Delta_Z\cdot Z$ is a proper
  $\sigma$-polyhedral divisor on a semiprojective variety $Y$. Since
  $X$ is endowed with an $\SL$-action of fiber type the cone $\sigma$
  admits an $\SL$-root. By Example~\ref{ex:toric-2} we can assume that
  $\sigma$ is the cone spanned in $N_\QQ$ by the vectors $(1,0)$ and
  $(r-1,r)$. In this case $e=(1,-1)$.

  By Remark~\ref{rk:fiber} (2) the $\sigma$-polyhedra $\Delta_Z$ is
  $v_Z+\sigma$, where $v_Z\in N_\QQ$. The divisor $\DD(e)$ is principal
  by Theorem~\ref{sec:sl-fiber} and is given by
  $$\DD(e)=\sum_Z \langle e,v_Z\rangle\cdot Z\,.$$
  Furthermore, by Corollary~\ref{cor:AH} we can assume that $\DD(e)=0$,
  so that for every $Z$
  $$v_Z=\alpha_Z(1,1)\quad \mbox{for some}\quad \alpha_Z\in \QQ\,.$$
  Letting $H=\DD((1,0))=\sum_Z\alpha_Z\cdot Z$ we obtain that
  $\DD=\Delta\cdot H$, where $\Delta=(1,1)+\sigma$. Recall that the
  divisor $H$ is semiample and big but not necessarily
  ample. Nevertheless, by \cite[Proposition 3.3]{Dem88} the
  combinatorial data $(Y,\DD)$ may be chosen so that $H$ is ample.
\end{proof}

The following theorem is a direct consequence of
Proposition~\ref{lm:SU} and Proposition~\ref{lm:T2-ex}.

\begin{theorem} \label{th:special} %
  Every normal affine variety $X$ of dimension $k+2$ endowed with a
  special $\SL$-action is uniquely determined by a positive integer
  $r$, a semiprojective variety $Y$ of dimension $k$, and a linear
  equivalence class $[H]$ of ample $\QQ$-Cartier divisors on $Y$.
\end{theorem}

\begin{remark} \label{rk:special-easy}
  The variety $X$ can be recovered from the data in
  Theorem~\ref{th:special} as follows. Let $\sigma$ be the cone
  spanned in $N_\QQ\simeq \QQ^2$ by the rays $(1,0)$ and $(r-1,r)$,
  $r>0$, and let $B_s=H^0(Y,\OO_Y(sH))$. Then $X$ is equivariantly
  isomorphic to $\spec A$, where $A$ is the $M$-graded algebra
  $$A=\bigoplus_{m\in \sigma^\vee_M}A_m\chi^m, \quad\mbox{such that}\quad
  A_m=B_{m_1+m_2}\,.$$

  The pair $(Y,[H])$ defines, via Demazure's construction
  \cite{Dem88}, a $\TT^1$-variety $W$ of dimension $k+1$. Then the
  variety $W$ with the new non-effective $\TT^1$-action given by
  $\TT^1\rightarrow \TT^1$, $t\mapsto t^r$, is nothing but $\spec
  \KK[X]^{U_+}$ endowed with the action of the maximal torus
  $T\subseteq \SL$. The corresponding non-effective $\ZZ$-grading on
  $\KK[X]^{U_+}$ is given by
  $$\KK[X]^{U_+}=\bigoplus_{i\in\ZZ_{\geq0}}B_i t^{ri}\,.$$
\end{remark}

From this classification we obtain the following corollary. 

\begin{corollary} \label{cor-sp-tor}
  Let $X$ be an affine toric variety endowed with a special
  $\SL$-action. If the canonical torus $\TT^2$ of the special action
  is contained in the big torus, then the $\SL$-action is normalized
  by the big torus.
\end{corollary}
\begin{proof}
  By Proposition~\ref{lm:SU} and Proposition~\ref{lm:T2-ex} we can assume
  that $X$ regarded as $\TT^2$-variety is given by the combinatorial
  data $X=X[Y,\DD]$, where $Y$ is a normal semiprojective variety $Y$
  and $\DD$ is the proper $\sigma$-polyhedral divisor given by
  $$\DD=\sum_Z \Delta_Z\cdot Z\,,$$
  where $\sigma$ is the cone spanned in $N_\QQ\simeq \QQ^2$ by the
  vectors $(1,0)$ and $(r-1,r)$, $\Delta_Z=\alpha_Z(1,1)+\sigma$, and
  $\sum_Z \alpha_Z\cdot Z$ is an ample $\QQ$-Cartier divisor on
  $Y$. In this case, the $\SL$-root of the cone $\sigma$ is
  $e=(1,-1)$. Furthermore, the $\SL$-action of fiber type
  corresponding to $e$ is unique.

  Since $X$ is toric and $\TT^2$ is a subtorus of the big torus, by
  \cite[Section 11]{AlHa06} we can assume that $Y$ is the toric
  variety given by a fan $\Sigma\subseteq \tN_\QQ$ and that $\DD$ is
  supported in the toric divisors of $Y$. Denote by $Z_\rho$ the toric
  divisor corresponding to a ray $\rho\in\Sigma(1)$. In this case, the
  $X$ is the toric variety given by the cone $\widetilde{\sigma}$ in
  $N_\QQ\oplus\tN_\QQ$ spanned by
  $$(\sigma,\bar{0})\quad \mbox{and}\quad (\Delta_{Z_\rho},\rho)\quad\forall
  \rho\in \Sigma(1)\,.$$ %
  Hence, the rays of the cone $\widetilde{\sigma}$ are spanned by
  $$\nu_+=((1,0),\bar{0}),\quad \nu_-=((r-1,r),\bar{0}),
  \quad\mbox{and}\quad
  \nu_\rho=(\alpha_{Z_\rho}(1,1),\rho)\,\quad\forall
  \rho\in\Sigma(1)\,.$$

  We claim that there exists an $\SL$-root $\widetilde{e}\in
  M\oplus\tM$ of the cone $\widetilde{\sigma}$ which restricted to
  $\sigma$ gives the $\SL$-root $e$ of $\sigma$. Indeed,
  $\widetilde{e}=(e,\bar{0})=((1,-1),\bar{0})$ is an $\SL$-root of the
  cone $\widetilde{\sigma}$ since the duality pairing between
  $\widetilde{e}$ and the rays of $\widetilde{\sigma}$ are
  $\langle\widetilde{e},\nu_+\rangle=1$,
  $\langle\widetilde{e},\nu_-\rangle=-1$, and
  $\langle\widetilde{e},\nu_\rho\rangle=0$ for all
  $\rho\in\Sigma(1)$.
\end{proof}

\begin{remark}
  In the case of a special $\SL$-action on a 3-dimensional toric
  variety $X$, by \cite{BeHa03} the canonical torus $\TT^2$ is
  conjugated to a subtorus of the big torus. So up to conjugation in
  $\Aut(X)$, every special $\SL$-action on $X$ is normalized by the
  big torus. In higher dimension, it is an open problem whether
  $\TT^2$ is conjugated to a subtorus of the big torus.
\end{remark}

\begin{corollary}
  Consider a special $\SL$-action on the affine space $\AF^s$. Assume
  that the action of the canonical torus $\TT^2$ on $\AF^s$ is
  linearizable. Then there exists an $\SL$-equivariant isomorphism
  $\AF^s\simeq\KK^2\oplus\KK^{s-2}$, where $\KK^2$ is the tautological
  $\SL$-module and the $\SL$-action on $\KK^{s-2}$ is identical.
\end{corollary}

\begin{proof}
  By Corollary~\ref{cor-sp-tor} we may assume that the $\SL$-action on
  $\AF^s$ is normalized by the torus of all diagonal matrices $\TT^s$
  and, moreover, this action is given by the $\SL$-root
  $(1,-1,0,\ldots,0)$.
\end{proof}

\section{Quasi-homogeneous $\SL$-threefolds}
\label{sec:popov}

In this section we study $\SL$-actions with an open orbit on a normal
affine threefold~$X$.

\subsection{$\SL$-threefolds via polyhedral divisors}

It is a byproduct of a classification due to Popov \cite{Pop73} that
most quasi-homogeneous $\SL$-threefolds admit the action of a two
dimensional torus making the $\SL$-action compatible (see
below). Hence, they can be classified with the methods of
Section~\ref{sec:T-var-hor}. In this section $\TT^2$ denotes an
algebraic torus of dimension 2, and so $\rank M=2$.

\begin{proposition} \label{th-popov} Let $X$ be an affine normal
  quasi-homogeneous $\SL$-threefold. Then $X$ admit the action of a
  two dimensional torus $\TT^2$ making the $\SL$-action compatible
  except in the case where $X$ is equivariantly isomorphic to $\SL/H$,
  with $H<\SL$ non-commutative and finite. Furthermore, $\TT^2=T\times
  R$, where $T$ is the maximal torus in $\SL$ and $R$ is a
  one-dimensional torus commuting with the $\SL$-action.
\end{proposition}

\begin{proof}
  See \cite{Pop73} or \cite[Ch. 3, p. 4.8]{Kra84}.
\end{proof}

In the following we restrict to the case where $X\not\simeq \SL/H$
with $H<\SL$ non-commutative and finite so that $X$ can be regarded as
a $\TT^2$-variety of complexity one. Up to conjugation, the only
finite commutative subgroups of $\SL$ are the cyclic groups
$$ \mu_r=\left\{
  \left(\begin{smallmatrix}
    \xi & 0 \\
    0 & \xi^{-1}
  \end{smallmatrix}\right)
  \mid \xi^r=1\right\},\quad r\in \ZZ_{>0}\,.$$

Popov's classification is given in terms of the order of the generic
stabilizer $r_X$ and the so-called height $h_X$. We propose to replace
the height by the slope $\hbar_X$. See Definitions \ref{def:slope} and
\ref{def:popov-h} for a precise definition. The main result of this
section is the following theorem. The invariants $r_X$, $h_X$ and
$\hbar_X$ are given in the table below but the computation of their
values will be performed after the proof of the theorem. In the table
we also give the number $N_X$ of $\SL$-orbits of $X$. This number is
only given for reference as it is not proved in the text.

\begin{theorem} \label{th:sl-emb} %
  Let $X$ be an affine normal quasi-homogeneous $\SL$-threefold. Then
  $X\not\simeq \SL/H$, with $H<\SL$ non-commutative and finite if and
  only if $X\simeq X[C,\DD]$ and the combinatorial data $(C,\DD)$ is
  as in the following table: %
  \smallskip
  \begin{center}
    \begin{tabular}{|c|c|c|c|c|c|c|}
      \hline
      $C$ & $\sigma$ & $\DD$ & $r_X$ & $h_X$ & $\hbar_X$ & $N_X$  \bigstrut\\
      \hline \hline 
      $\AF^1$ & $\{0\}$ & $\Delta_0[0]+\Delta_1[1]$ & $r$ & -- & -- & $1$  \bigstrut\\
      \hline
      $\AF^1$ & $\cone((1,1))$ & $\Delta_0[0]+\Delta_1[1]$ & $r$ & $1$ &
      $1$ & $2$ \bigstrut\\
      \hline
      $\PP^1$ & $\cone((a+1,a),(r+a-1,r+a))$ &
      $\Delta_0[0]+\Delta_1[1]+\Delta_\infty[\infty]$ & $r$ &
      $\dfrac{a}{a+r}$ & $\dfrac{a}{a+1}$ & $3$ \bigstrut\\
      \hline
    \end{tabular}
  \end{center}
  \medskip
  Here $a\in \QQ_{>0}$, $r\in \ZZ_{>0}$,
  \begin{align*}
    \Delta_0=\conv(0,(1,0))+\sigma,&\quad
    \Delta_1=\conv(0,(r-1,r))+\sigma, \quad \mbox{and}\quad
    \Delta_\infty=(a,a)+\sigma\quad a\in \QQ_{>0}\,.
  \end{align*}
  Furthermore, $X[C,\DD]$ is an $(\SL/\mu_r)$-embedding, and $X$ is an
  homogeneous space of $\SL$ if and only if $\sigma=\{0\}$.
\end{theorem}

Before proving this theorem, we need a preliminary result. Let
$X=X[C,\DD]$ for some $\sigma$-polyhedral divisor on a smooth
curve. Since the $\SL$-action has an open orbit, the general isotropy
group is finite, and so the $\SL$-action in non-special. Hence, by
Proposition~\ref{prop:non-special} the $\SL$-action is of horizontal
type, $C=\AF^1$ or $C=\PP^1$, and we may and will assume in the sequel
that $\DD$ is as in Theorem~\ref{th:hor}.

We will use the following general lemma to identify the
$\sigma$-polyhedral divisors in Theorem~\ref{th:hor} that give rise to
quasihomogeneous $\SL$-actions. Here the dimension of a domain $A$ is
the dimension of the algebraic variety $\spec A$.

\begin{lemma}
  Let $X$ be a normal affine variety endowed with a non-special
  $\SL$-action. Then $X$ has finite generic stabilizer if and only if
  $\dim\KK[X]^{\SL}\leq \dim X-3$.
\end{lemma}

\begin{proof}
  Assume first that $X$ has finite generic stabilizer. By Rosenlicht's
  Theorem, the transcendence degree of $\KK(X)^{\SL}$ equals the
  codimension of the general orbit \cite[Corollary 6.2]{Dol03},
  so $$\trdeg \KK(X)^{\SL}=\dim X-3\,.$$ Since $\fract
  \KK[X]^{\SL}\subseteq \KK(X)^{\SL}$, we have $\dim\KK[X]^{\SL}\leq
  \dim X-3$.

  Assume now that the generic stabilizer has positive dimension. Since
  the $\SL$-action is non-special, the generic stabilizer is
  one-dimensional and coincides either with $T$ or $N$. In both cases
  the subgroup has a finite index in its normalizer in $\SL$. By
  \cite[Corollaire 3]{Lun75}, we obtain that the general $\SL$-orbits
  are closed in $X$. Thus, they are separated by regular invariants
  and so $\dim \KK[X]^{\SL}=\dim X-2$.
\end{proof}

We now proceed to the proof of the main theorem in this section.

\begin{proof}[Proof of Theorem \ref{th:sl-emb}]
  
  We prove first the ``only if'' part.  Let $X=X[C,\DD]$ be an affine
  $\TT^2$-variety admitting a compatible $\SL$-action with an open
  orbit. We have $C=\AF^1$ or $C=\PP^1$ and we can assume that $\DD$
  is as in Theorem~\ref{th:hor}. Let $\partial_\pm$ be the homogeneous
  LNDs of horizontal type corresponding to the $U_\pm$-action on $X$
  and let $\pm e$ be the degree of $\partial_\pm$.

  Since $v_1^+(e)=-1$, $e$ is a primitive lattice vector, so up to 
  automorphism of the lattice $M$, we may and will assume
  $e=(1,-1)$. Assume for a moment that $C=\PP^1$. In this case the
  cone~$\sigma$ is full dimensional and so $\sigma^\vee$ is
  pointed. Hence $\pm e\notin \sigma^\vee$. This yields
  $e^\bot=\QQ(1,1)$ intersects with $\sigma$ only once and so
  $\Delta_\infty=(a,a)+\sigma$.

  We have $\KK[X]^{U_\pm}=\ker\partial_\pm$. Since $\SL$ is generated
  by $U_\pm$ as a group, we have $\KK[X]^{\SL}=\ker\partial_+\cap
  \ker\partial_-$. Hence, the compatible $\SL$-action on $X[C,\DD]$
  has an open orbit if and only if
  $$\ker\partial_+\cap \ker\partial_-=\KK\,.$$

  By Lemma~\ref{lm:kernel}, if $\deg \DD|_{\AF^1}$ has only two
  vertices then $\ker\partial_+\cap
  \ker\partial_-\supsetneq\KK$. Hence the second family of
  $\sigma$-polyhedral divisors in Theorem~\ref{th:hor} does not give
  quasihomogeneous $\SL$-threefolds. In the following, we assume that
  $\DD$ is as in the first family in Theorem~\ref{th:hor}.

  Up to an automorphism of the lattice $N$, we can assume
  $v_0^-=(1,0)$ and $v_1^+=(r-1,r)$ with $r\in \ZZ_{\geq 0}$. If $r=0$
  then again $\deg\DD|_{\AF^1}$ has only two vertices and so
  $\ker\partial_+\cap \ker\partial_-\supsetneq\KK$. Hence, $r\geq 1$. This
  shows that $\Delta_0$ and $\Delta_1$ have the form given in the
  theorem.

  It only remains to find the tail cone $\sigma$. Let $C=\AF^1$. In
  this case, $\pm e\in \sigma^\vee$ and so $\sigma=\{0\}$ or
  $\sigma=\cone((1,1))$. If $C=\PP^1$, we let
  $\sigma=\cone(\rho_1,\rho_2)$. Since $\pm e\notin \sigma^\vee$, by
  Lemma~\ref{hor-int} we have
  $$\deg \DD\cap \rho_1\neq \emptyset, \quad\mbox{and}\quad
  \deg \DD\cap \rho_2\neq \emptyset\,.$$ %
  This yields $\rho_1=\cone(a+1,a)$ and $\rho_2=\cone(r+a-1,r+a)$ with
  $a>0$. This proves the ``only if'' part.

  Now, let $X=X[C,\DD]$ be as in the theorem. By Theorem~\ref{th:hor},
  a simple verification shows that $X$ admits an $\SL$-action and by
  Lemma~\ref{lm:kernel} we have $\ker\partial_+\cap
  \ker\partial_-=\KK$. Hence the $\SL$-action has finite generic
  stabilizer and so $X$ is a quasi-homogeneous $\SL$-threefold.

  The last assertion of the theorem is shown in Lemmas~\ref{lm:sl-Zr}
  and \ref{lm:rx} below.
\end{proof}

\subsection{Parameters}

In the remaining of this section we define and compute the parameters
$r_X$, $h_X$, and $\hbar_X$ given in the table in
Theorem~\ref{th:sl-emb}. First, we give a geometric interpretation of
the parameter $r_X$.

\begin{lemma} \label{lm:sl-Zr} Let $X=X[C,\DD]$ be as in
  Theorem~\ref{th:sl-emb}. Then $X$ is equivariantly isomorphic to the
  homogeneous space $\SL/\mu_r$ if and only if $C=\AF^1$ and
  $\sigma=\{0\}$.
\end{lemma}

\begin{proof}
  Assume that $X$ is a homogeneous space. The statement is equivalent
  to the fact that the general orbit of the acting torus is
  closed. Let us consider $G=\SL \times R$, where $R$ is the
  one-dimensional torus commuting with $\SL$, see
  Proposition~\ref{th-popov} so that $X$ is a homogeneous space of
  $G$. Then the acting torus $\TT^2=T \times R$ is a reductive
  subgroup of $G$. By \cite{Lun72} the general $\TT^2$-orbit on $X$ is
  closed.

  To complete the proof, we only need to show that $r$ in the
  definition of $\Delta_1$ is the order of the generic stabilizer of
  the $\SL$-action on $X$. By \cite[II.3.1, Satz 3]{Kra84}, the
  algebra $\KK[\SL]$ seen as $\SL\times\SL$-module has the isotypic
  decomposition
  $$\KK[\SL] \simeq \bigoplus_{d\ge 0} V(d) \otimes V(d)\,,$$
  where $V(d)$ is the simple $\SL$-module of binary forms of degree
  $d$.  The one-dimensional subtorus $R$ commuting with the (left)
  $\SL$-action may be identified with the maximal torus in the second
  (right) $\SL$. Then, the homogeneous space $\SL/\mu_r$ is obtained
  as the quotient of $\SL$ by the cyclic subgroup of order $r$ in
  $R$. So simple $\SL$-submodules in $\KK[\SL/\mu_r]$ have the form
  $V(d) \otimes w$, where $w$ runs through $R$-weight vectors of
  $V(d)$ whose weight is divisible by $r$.

  Now, the subalgebra of $U_+$-invariants of $\KK[\SL/\mu_r]$ is
  spanned by the elements $v \otimes w \in V(d)\otimes w$, where $v$
  is highest weight vector in $V(d)$. Let $T$ be the maximal torus in
  the (left) $\SL$-action. We have shown that the order $r$ of the
  generic stabilizer $r$ is the minimal integer such that
  $\ker\partial_+$ contains a $T$-weight vector of weight $r$ which is
  not $R$-invariant.

  We return now to the combinatorial data $X=X[C,\DD]$. Since
  $e=(1,-1)$, the grading given by $R$ corresponds to the ray
  $p_R=(1,1)$, and by the proof of Lemma~\ref{lm:hor-DD} the grading
  given by $T$ corresponds to the ray $p_T=v_0^- - v_1^+ = (1,0) -
  (r-1,r) = (-r+2, r)$.  By Theorem~\ref{lnd-hor}, the cone of the
  semigroup algebra $\ker\partial_+$ is dual to
  $\omega=\cone((-1,0),(r-1,r))$. The semigroup $\omega^\vee_M$ is
  spanned by $m_1=(0,1)$, $m_2=(-r,r-1)$, and $m_3=(-1,1)$ and we have
  $$\langle m_1,p_T \rangle=\langle m_2,p_T\rangle=r,
  \quad \langle m_3,p_T \rangle=2, \quad\mbox{and}\quad \langle
  m_3,p_R \rangle=0\,.$$ %
  Hence, the minimal weight such that the $\ker\partial_+$ contains a
  $T$-weight vector which is not $R$-invariant is $r$, and the lemma
  follows.
\end{proof}

\begin{lemma} \label{lm:rx} %
  Let $X=X[C,\DD]$ be as in Theorem~\ref{th:sl-emb}. Then $r=r_X$ is
  the order of the generic stabilizer and $X$ is an
  $(\SL/\mu_r)$-embedding.
\end{lemma}

\begin{proof}
  If $C=\AF^1$, $\sigma=\{0\}$, and $r\geq 1$, then the result follows
  from Lemma~\ref{lm:sl-Zr}. Let now $X=X[C,\DD]$ be as in
  Theorem~\ref{th:sl-emb} with $\sigma\neq \{0\}$. By
  \cite[Theorem~17]{AIPSV} there is a $\TT^2$-equivariant open
  embedding $\SL/\mu_r\hookrightarrow X$. Hence, the lemma follows
  from the homogeneous case.
\end{proof}

In the following we assume that $X$ is not equivariantly isomorphic to
a homogeneous space. Let $r$ be the order of the generic stabilizer of
$X$. The open embedding $\SL/\mu_r\hookrightarrow X$ induces an
inclusion of the algebras of $U_+$-invariants
$\KK[X]^{U_+}\hookrightarrow \KK[\SL/\mu_r]^{U_+}$. Both these
algebras are semigroup algebras. Moreover, the cones of these
semigroup algebras share a common ray in~$M_\QQ$. This ray will be
denoted by $\rho_{U_+}$.

Let $\omega\subseteq M_\QQ\simeq \QQ^2$ be a full dimensional cone and
let $\rho$ be one of its rays. It is well known that, up to 
automorphism of the lattice $M$, we can assume $\rho=\cone((1,0))$ and
$\omega=\cone((1,0),(b,c))$ with $1\leq b\leq c$ and $\gcd(b,c)=1$. We
define the slope of $\omega$ with respect to $\rho$ as
$\nicefrac{b}{c}\in \QQ\cap (0,1]$.

\begin{definition} \label{def:slope} %
  Let $X$ be a non-homogeneous quasi-homogeneous $\SL$-threefold.  We
  define the \emph{slope} $\hbar_X$ of $X$ as the slope of the cone of
  the ring of $U_+$-invariants with respect to the ray $\rho_{U_+}$.
\end{definition}

\begin{remark}
  This definition does not coincide with the height defined by Popov
  and used in the literature \cite{Pop73,Kra84,Gai08,BaHa08}. The
  height will be introduced below and denoted by the plain letter
  $h_X$. We will also show the relation between slope and height. The
  main motivation to use a different definition is that the results
  have simpler statements.
\end{remark}

Let $X=X[C,\DD]$ be as in Theorem~\ref{th:sl-emb} and assume that $X$
is not a homogeneous space. If $C=\AF^1$ and $\sigma=\cone((1,1))$,
then by Theorem~\ref{lnd-hor} the cone of the ring of $U_+$-invariants
$\KK[X]^{U_+}$ is given by $\cone((0,1)(-1,1))$ and so the slope of
$X$ is $\hbar_X=1$. Assume now that $C=\PP^1$. In this case, by
Theorem~\ref{lnd-hor}, the cone of the ring $\KK[X]^{U_+}$ is given by
$\cone((0,1),(-a,a+1))$ and the common ray of the cones of
$\KK[X]^{U_+}$ and $\KK[\SL/\mu_r]^{U_+}$ is spanned by the lattice
vector $(0,1)$. Hence, the slope of $X$ is
\begin{align} \label{eq:hbar}
\hbar_X=\frac{a}{a+1}\in \QQ\cap(0,1)\,.
\end{align}
Since the function defining $\hbar_X$ in terms of $a$ is one to one,
we have the following corollary.

\begin{corollary}
  Two non-homogeneous quasi-homogeneous $\SL$-threefolds $X$ and $X'$
  are equivariantly isomorphic if and only if $r_X=r_{X'}$ and
  $\hbar_X=\hbar_{X'}$.
\end{corollary}

In the following corollary we give a criterion for a quasihomogeneous
$\SL$-threefold to be toric. This result is also given in \cite{Gai08}
and \cite{BaHa08} in terms of the height of $X$.

\begin{corollary} \label{cor:gaif} %
  Let $X$ be a quasi-homogeneous $\SL$-threefold. Then $X$ is a toric
  variety if and only if $X$ is non-homogeneous and
  $\hbar_X=\nicefrac{p}{p+1}$ for some $p\in \ZZ_{>0}$.
\end{corollary}

\begin{proof}
  Let $X\simeq X[C,\DD]$ with $C=\AF^1$ or $C=\PP^1$ and $\DD$ as in
  Theorem~\ref{th:sl-emb}. By Corollaries~\ref{cor:AH} and
  \ref{AH-toric}, we obtain that $X$ is toric if and only if $C=\PP^1$
  and $a$ is an integer. Let now $\hbar=\nicefrac{p}{q}$ with
  $\gcd(p,q)=1$ and $p,q\geq 0$. By \eqref{eq:hbar} we have
  $a=\nicefrac{p}{q-p}$ and so the result follows.
\end{proof}

\begin{remark}
  In Corollary~\ref{cor:gaif}, the $\SL$-action is not compatible with
  the big torus, since otherwise the $\SL$-action would be special.
\end{remark}

\subsection{Relation between slope and height}
\label{sec:popov-height}

Let $X$ be a non-homogeneous quasi-homogeneous $\SL$-threefold.

\begin{definition}\label{def:popov-h}
  The \emph{height} $h_X$ of $X$ is defined as follows. If $r_X=1$
  then the height of $X$ is the same as the slope of $X$ i.e.,
  $h_X=\hbar_X$. If $r_X>1$ then there is a unique non-homogeneous
  quasi-homogeneous $\SL$-threefold $X'$ with $r_{X'}=1$ such that
  $X=X'/\mu_r$ (see \cite{Pop73} or \cite[III.4.9, Satz
  1]{Kra84}). The height of $X$ is defined as the slope of $X'$ i.e.,
  $h_X=\hbar_{X'}$.
\end{definition}

In this section we compute the height of $X$ in terms of the slope and
the order of the generic stabilizer. We also state
Corollary~\ref{cor:gaif} in terms of the height.

Assume that $r_X> 1$ and let $X'$ be as in
Definition~\ref{def:popov-h}. We let $M$, $N$, $\sigma$, $C$,
$\DD=\sum \Delta_z\cdot z$ and $M'$, $N'$, $\sigma'$, $C'$, $\DD'=\sum
\Delta'_z\cdot z$ be the combinatorial data of $X$ and $X'$,
respectively. By Definition~\ref{def:popov-h} we have
\begin{align*}
  &\Delta_0=\conv(0,(1,0))+\sigma,\quad
  \Delta_1=\conv(0,(r-1,r))+\sigma,\\
  &\Delta_0'=\conv(0,(1,0))+\sigma',\quad
  \Delta_1'=\conv(0,(0,1))+\sigma'\,.
\end{align*}

The morphism $\varphi: X'\rightarrow X$ is given by the quotient by
$\mu_r$ contained in the $\TT^2$ acting on $X'$.  Hence, the morphism
$\varphi$ is given by a morphisms $\varphi_*:N'\rightarrow N$ of
lattices. Hence, $C\simeq C'$. Furthermore, since the morphism
$\varphi_*$ sends $\Delta_0'$ into $\Delta_0$ and $\Delta_1'$ into
$\Delta_1$ we have that $\varphi_*$ is given by
$$(1,0)\mapsto (1,0),\quad\mbox{and}\quad (0,1)\mapsto (r-1,r)\,.$$

If $C=\AF^1$, then $C'\simeq \AF^1$ and so $h_X=\hbar_{X'}=1$. Assume
that $C=\PP^1$ and let $\Delta_\infty=(a,a)+\sigma$. In this case
$C'\simeq \PP^1$ and $\Delta_\infty'=\tfrac{1}{r}(a,a)+\sigma'$. Now
\eqref{eq:hbar} yields
\begin{align} \label{eq:h}
h_X=\hbar_{X'}=\frac{a}{a+r}\in \QQ\cap (0,1)\,.
\end{align}

Relations \eqref{eq:hbar} and \eqref{eq:h} imply the following
corollary.

\begin{corollary}
  Let $X$ be a non-homogeneous quasi-homogeneous $\SL$-threefold. Then
  $$h_X=\frac{\hbar_X}{r_X-(r_X-1)\hbar_X}\,.$$
\end{corollary}

Finally, a direct computation shows that in terms of the height
Corollary~\ref{cor:gaif} takes the form as in \cite{Gai08,BaHa08}.

\begin{corollary} \label{cor:gaif-2} %
  Let $X$ be a non-homogeneous quasi-homogeneous $\SL$-threefold and
  let $h_X=\nicefrac{p}{q}$, where $\gcd(p,q)=1$ and $p,q>0$. Then $X$
  is a toric variety if and only if $q-p$ divides $r$.
\end{corollary}

\subsection{Generically transitive $\SL\times \TT^s$-action}

Consider now the reductive group $G=\SL\times \TT^s$ for some $s\in
\ZZ_{\geq 0}$. Any action of this group on a normal affine variety is
compatible with the action of the torus $\TT=T\times \TT^s$, where
$T\subseteq \SL$ is a maximal torus. The results of
Section~\ref{sec:T-var-hor} may be regarded as a classification of
generically transitive $G$-actions under the assumption that the
complexity of the corresponding $\TT$-action does not exceed one. In
Section~\ref{sec:popov} we deal with the case $s=1$ and because of
Proposition~\ref{th-popov} this yields a classification of generically
transitive $G$-actions with $s=0$. The following example shows that
our techniques does not allow to describe all generically transitive
$G$-actions with $s=1$.

\begin{example}
  Let $G=\SL\times \KK^*$ and $X=V_3=\langle x^3,x^2y,xy^2,y^3\rangle$
  be a simple $\SL$-module of binary forms of degree 3, where $\KK^*$
  acts by scalar multiplication. The module $V_3$ contains a
  1-parameter family of general $\SL$-orbits, namely
  $\SL\cdot(\alpha(x^3+y^3))$, $\alpha\in \KK\setminus\{0\}$. Thus,
  $G$ acts on $V_3$ with an open orbit isomorphic to $G/\mu_2$. The
  torus $\TT=T\times \KK^*$ acts on $X$ with complexity two. Assume
  that $\TT$ may be extended by a torus $R$ commuting with $G$. Then
  $R$ commutes with $\KK^*$ and its action descends to the
  projectivization $\PP(V_3)$. Then $R$ maps to a subtorus of
  $\PGL_4$. Since $V_3$ is simple, by Schur's Lemma there are no
  non-identity elements in $\PGL_4$ commuting with the image of $\SL$.
\end{example}

\section*{Appendix: The commutator formulas}
\label{appendix}

In this appendix we prove the commutator formulas~\eqref{commutator-1}
and \eqref{commutator-2} used in Section~\ref{sec:T-var-hor}. The
computations are routine, but cumbersome, so we put them in an
appendix instead of the main body of the paper for easy reading.

We keep the notation as in Section~\ref{sec:T-var-hor}. The main idea
is that from \eqref{eq:LND-} and \eqref{eq:LND+} we can obtain
formulas for $\partial_\pm(\chi^m)$ and for $\partial_\pm(t)$ by
applying the Leibniz rule. Then we use this formulas to compute the
commutator $\delta=[\partial_+,\partial_-]$. 

A simple evaluation of \eqref{eq:LND-} and \eqref{eq:LND+} yields
$$\partial_-(t)=d^-\cdot\chi^{-e}\cdot(q')^{-1}\cdot q^{1+s^-},\quad 
\partial_-(\chi^m)=d^-\cdot v_{z_0^-}^-(m)\cdot
\chi^{m-e}\cdot q^{s^-}\quad\mbox{and}, $$
$$\partial_+(t)=d^+\cdot \varphi^e\cdot \chi^{e}\cdot t^{1+s^+}\,.$$
Evaluating now \eqref{eq:LND+} for $r=0$ we obtain
$$(\varphi^m)'\cdot\chi^m\cdot\partial_+(t)+\varphi^m\cdot \partial_+(\chi^m)=
d^+\cdot v_0^+(m)\cdot \varphi^{m+e}\cdot \chi^{m+e}\cdot t^{s^+}\,.$$
And using the definition of $\alpha_m$ we obtain
$$\partial_+(\chi^m)=
d^+\cdot(v_0^+(m)-\alpha_m)\cdot \varphi^{e}\cdot \chi^{m+e}\cdot
t^{s^+}\,.$$ 

\medskip\noindent\textbf{Formula \eqref{commutator-1}.}
We compute first $\partial_+\partial_-(t)$ and
$\partial_-\partial_+(t)$.

\begin{align*}
  \partial_+\partial_-(t)&=\partial_+
  \left(d^-\cdot\chi^{-e}\cdot(q')^{-1}\cdot q^{1+s^-}\right)\\
  &=d^+d^-\varphi^e
  t^{s^+}q^{s^-}\left(\left(\alpha_e-v_0^+(e)\right)\cdot\frac{q}{q'}-
    \frac{q''q t}{(q')^2}+(1+s^-)\cdot t \right), \quad\mbox{and} \\
  \partial_-\partial_+(t)&=\partial_-
  \left(d^+\cdot\chi^{e}\cdot t^{1+s^+}\right)\\
  &=d^+d^-\varphi^e t^{s^+}q^{s^-}\left(v_{z_0^-}^-(e)\cdot
    t+\left(1+s^++\alpha_e\right)\cdot\frac{q}{q'} \right)\,.
\end{align*}

Hence, the commutator is given by
$$\delta(t)=d^+d^-\varphi^e
t^{s^+}q^{s^-}\left(\left(1+s^--v_{z_0^-}^-(e)\right)\cdot t
  -\left(1+s^++v_0^+(e)\right)\cdot\frac{q}{q'}-
  \frac{q''q t}{(q')^2} \right)\, $$ %
and \eqref{commutator-1} follows since
$s^+=-\nicefrac{1}{d^+}-v_{0}^+(e)$ and
$s^-=-\nicefrac{1}{d^-}+v_{z_0^-}^-(e)$.

\medskip\noindent\textbf{Formula \eqref{commutator-2}.} In this case
we have $z_0^\pm=0$, $z_\infty^\pm=\infty$ and $d^\pm=d$. Hence

$$\partial_-(t)=d\cdot\chi^{-e}\cdot t^{1+s^-},\quad 
\partial_-(\chi^m)=d\cdot v_{0}^-(m)\cdot \chi^{m-e}\cdot
t^{s^-}\quad \partial_+(t)=d \cdot \varphi^e\cdot \chi^{e}\cdot
t^{1+s^+}\quad\mbox{and}\,.$$
$$\partial_+(\chi^m)=
d\cdot(v_0^+(m)-\alpha_m)\cdot \varphi^{e}\cdot \chi^{m+e}\cdot
t^{s^+}\,.$$ 

This yields
\begin{align*}
&\partial_+(\chi^mt^r)=d\varphi^e\cdot
(v_0^+(m)+r-\alpha_m)\cdot\chi^{m+e}t^{r+s^+}, \quad\mbox{and}\\
&\partial_-(\chi^mt^r)=d\cdot
(v_0^-(m)+r)\cdot\chi^{m-e}t^{r+s^-}
\end{align*}
Recall that $s^+=-\nicefrac{1}{d}-v_0^+(e)$,
$s^-=-\nicefrac{1}{d}+v_0^-(e)$, $v_0=v_0^--v_0^+$ and
$\nu=v_0(e)-\nicefrac{1}{d}=s^++s^-+\nicefrac{1}{d}$. Now a direct
computation yields
\begin{align*}
\partial_+\partial_-(\chi^mt^r)=d^2\chi^e&\cdot\left(v^-_0(m)+r\right)
\cdot\left(v^+_0(m)+\nu+r-\alpha_m+\alpha_e\right)\cdot\chi^m
t^{\nu-\nicefrac{1}{d}}, \quad\mbox{and}\\
\partial_-\partial_+(\chi^mt^r)=d^2\chi^e&\cdot\left(-t\alpha_m'+
  \alpha_e(v^+_0(m)+r-\alpha_m)+\right.\\
&\left.+(v^+_0(m)+r-\alpha_m)(v^-_0(m)+\nu+r) \right)\cdot\chi^m
t^{\nu-\nicefrac{1}{d}}\,.
\end{align*}

Formula \eqref{commutator-2} follows by computing
$\delta(\chi^mt^r)=\partial_+\partial_-(\chi^mt^r)-
\partial_-\partial_+(\chi^mt^r)$.

\bibliographystyle{alpha} \bibliography{math}

\newcommand{\etalchar}[1]{$^{#1}$}
\begin{thebibliography}{KKMS73}

\bibitem[AH06]{AlHa06}
Klaus Altmann and J{\"u}rgen Hausen.
\newblock Polyhedral divisors and algebraic torus actions.
\newblock {\em Math. Ann.}, 334(3):557--607, 2006.

\bibitem[AHS08]{AHS08}
Klaus Altmann, J{\"u}rgen Hausen, and Hendrik S{\"u}ss.
\newblock Gluing affine torus actions via divisorial fans.
\newblock {\em Transform. Groups}, 13(2):215--242, 2008.

\bibitem[AIP{\etalchar{+}}11]{AIPSV}
Klaus Altmann, Nathan{ O}wen Ilten, Lars Petersen, Hendrik S{\"u}ss, and Robert
  Vollmert.
\newblock The geometry of {T}-varieties, 2011.
\newblock arXiv:1102.5760v1 [math.AG], 42p. (To appear in Contributions to
  Algebraic Geometry, IMPANGA Lecture Notes).

\bibitem[Arz97]{Arz97}
Ivan~V. Arzhantsev.
\newblock On {${\rm SL}_2$}-actions of complexity one.
\newblock {\em Izv. Ross. Akad. Nauk Ser. Mat.}, 61(4):3--18, 1997.

\bibitem[BH03]{BeHa03}
Florian Berchtold and J{\"u}rgen Hausen.
\newblock Demushkin's theorem in codimension one.
\newblock {\em Math. Z.}, 244(4):697--703, 2003.

\bibitem[BH08]{BaHa08}
Victor Batyrev and Fatima Haddad.
\newblock On the geometry of {${\rm SL}(2)$}-equivariant flips.
\newblock {\em Mosc. Math. J.}, 8(4):621--646, 846, 2008.

\bibitem[Dem70]{Dem70}
Michel Demazure.
\newblock Sous-groupes alg{\'e}briques de rang maximum du groupe de {C}remona.
\newblock {\em Ann. Sci. {\'E}cole Norm. Sup. (4)}, 3:507--588, 1970.

\bibitem[Dem88]{Dem88}
Michel Demazure.
\newblock Anneaux gradu\'es normaux.
\newblock In {\em Introduction \`a la th\'eorie des singularit\'es, {II}},
  volume~37 of {\em Travaux en Cours}, pages 35--68. Hermann, Paris, 1988.

\bibitem[Dol03]{Dol03}
Igor Dolgachev.
\newblock {\em Lectures on invariant theory}, volume 296 of {\em London
  Mathematical Society Lecture Note Series}.
\newblock Cambridge University Press, Cambridge, 2003.

\bibitem[Fre06]{Fre06}
Gene Freudenburg.
\newblock {\em {Algebraic theory of locally nilpotent derivations.}}
\newblock {Encyclopaedia of Mathematical Sciences 136. Invariant Theory and
  Algebraic Transformation Groups 7. Berlin: Springer}, 2006.

\bibitem[Ful93]{Fu93}
William Fulton.
\newblock {\em Introduction to toric varieties}, volume 131 of {\em Annals of
  Mathematics Studies}.
\newblock Princeton University Press, Princeton, NJ, 1993.
\newblock The William H. Roever Lectures in Geometry.

\bibitem[FZ05]{FlZa05b}
Hubert Flenner and Mikhail Zaidenberg.
\newblock Locally nilpotent derivations on affine surfaces with a
  {$\mathbb{C}^*$}-action.
\newblock {\em Osaka J. Math.}, 42(4):931--974, 2005.

\bibitem[Ga{\u\i}08]{Gai08}
Sergey~A. Ga{\u\i}fullin.
\newblock Affine toric {${\rm SL}(2)$}-embeddings.
\newblock {\em Mat. Sb.}, 199(3):3--24, 2008.

\bibitem[KKMS73]{KKMS73}
G.~Kempf, F.~Knudsen, D.~Mumford, and B.~Saint$\mbox{-Donat}$.
\newblock {\em Toroidal embeddings. {I}}.
\newblock Lecture Notes in Mathematics, Vol. 339. Springer-Verlag, Berlin,
  1973.

\bibitem[Kra84]{Kra84}
Hanspeter Kraft.
\newblock {\em Geometrische {M}ethoden in der {I}nvariantentheorie}.
\newblock Aspects of Mathematics, D1. Friedr. Vieweg \& Sohn, Braunschweig,
  1984.

\bibitem[Lie10a]{Lie10}
Alvaro Liendo.
\newblock Affine $\mathbb{T}$-varieties of complexity one and locally nilpotent
  derivations.
\newblock {\em Transform. Groups}, 15(2):389--425, 2010.

\bibitem[Lie10b]{Lie10b}
Alvaro Liendo.
\newblock $\mathbb{G}_{\mathrm{a}}$-actions of fiber type on affine
  $\mathbb{T}$-varieties.
\newblock {\em J. Algebra}, 324:3653--3665, 2010.

\bibitem[Lun72]{Lun72}
Domingo Luna.
\newblock Sur les orbites ferm\'ees des groupes alg\'ebriques r\'eductifs.
\newblock {\em Invent. Math.}, 16:1--5, 1972.

\bibitem[Lun75]{Lun75}
Domingo Luna.
\newblock Adh\'erences d'orbite et invariants.
\newblock {\em Invent. Math.}, 29(3):231--238, 1975.

\bibitem[Pop73]{Pop73}
Vladimir~L. Popov.
\newblock Quasihomogeneous affine algebraic varieties of the group {${\rm
  SL}(2)$}.
\newblock {\em Izv. Akad. Nauk SSSR Ser. Mat}, 37:792--832, 1973.

\bibitem[Pop86]{Pop86}
Vladimir~L. Popov.
\newblock Contractions of actions of reductive algebraic groups.
\newblock {\em Mat. Sb. (N.S.)}, 130(172)(3):310--334, 1986.

\bibitem[Tim97]{Tim97}
Dmitri~A. Timashev.
\newblock Classification of {$G$}-manifolds of complexity {$1$}.
\newblock {\em Izv. Ross. Akad. Nauk Ser. Mat.}, 61(2):127--162, 1997.

\bibitem[Tim08]{Tim08}
Dmitri~A. Timashev.
\newblock Torus actions of complexity one.
\newblock In {\em Toric topology}, volume 460 of {\em Contemp. Math.}, pages
  349--364. Amer. Math. Soc., Providence, RI, 2008.

\end{thebibliography}
\end{document}